\documentclass[11pt]{article}

\usepackage{amsmath,amsthm,amssymb,bbm,multicol,tikz,subfig,float,titlesec}

\usepackage[margin=1.2in]{geometry}

\newtheorem{theorem}{Theorem}
\newtheorem{lemma}{Lemma}
\newtheorem{corollary}{Corollary}
\theoremstyle{definition}
\newtheorem{example}{Example}
\newtheorem{note}{Note}
\newtheorem{definition}{Definition}
\newtheorem{problem}{Problem}
\newtheorem{question}{Question}
\newtheorem{conjecture}{Conjecture}

\DeclareMathOperator{\rep}{Rep}

\titleformat{\section}{\large\bfseries}{\thesection}{1em}{}
\titleformat{\subsection}{\bfseries}{\thesubsection}{1em}{}

\title{Prime decomposition of modular tensor \\ categories  of local modules of Type D}
\date{\today}
\author{Andrew Schopieray}

\begin{document}

\maketitle

\begin{abstract}\footnotesize
\noindent Let $\mathcal{C}(\mathfrak{g},k)$ be the unitary modular tensor categories arising from the representation theory of quantum groups at roots of unity for arbitrary simple finite-dimensional complex Lie algebra $\mathfrak{g}$ and positive integer levels $k$.  Here we classify nondegenerate fusion subcategories of the modular tensor categories of local modules $\mathcal{C}(\mathfrak{g},k)_R^0$ where $R$ is the regular algebra of Tannakian $\text{Rep}(H)\subset\mathcal{C}(\mathfrak{g},k)_\text{pt}$.  For $\mathfrak{g}=\mathfrak{so}_5$ we describe the decomposition of $\mathcal{C}(\mathfrak{g},k)_R^0$ into prime factors explicitly and as an application we classify relations in the Witt group of nondegenerately braided fusion categories generated by the equivalency classes of $\mathcal{C}(\mathfrak{so}_5,k)$ and $\mathcal{C}(\mathfrak{g}_2,k)$ for $k\in\mathbb{Z}_{\geq1}$.

\end{abstract}

\begin{section}{Introduction}
\par For each simple finite-dimensional complex Lie algebra $\mathfrak{g}$ and positive integer level $k$, there exists a unitary modular tensor category $\mathcal{C}(\mathfrak{g},k)$ which can be realized as a semisimple quotient of the representation theory of the quantum group $\mathcal{U}_q(\mathfrak{g})$ for a particular root of unity $q$ (see Section \ref{sec:lie}).  At distinguished levels one may construct a new modular tensor category $\mathcal{C}(\mathfrak{g},k)_R^0$ (simple current extension, de-equivariantization, condensation) consisting of local modules over a nontrivial connected \'etale algebra $R\in\mathcal{C}(\mathfrak{g},k)_\text{pt}$.  Theorem \ref{th:simplicity} states with few exceptions, the nondegenerate fusion subcategories of $\mathcal{C}(\mathfrak{g},k)_R^0$ are pointed fusion subcategories or their relative commutants, and identifies restrictive conditions for the existence of exceptional degenerate fusion subcategories.  As a consequence, the modular tensor categories $\mathcal{C}(\mathfrak{g},k)_R^0$ can be decomposed into prime factors using M\"uger's decomposition \cite[Theorem 4.2]{mug1}.  We supply one additional application of Theorem \ref{th:simplicity}: Section \ref{wittrel} describes a complete classification of relations in the Witt group of nondegenerate braided fusion categories generated by the classes $[\mathcal{C}(\mathfrak{g},k)]$ where rank\,$\mathfrak{g}\leq2$.

\par Fusion subcategories are an obligatory invariant of fusion categories.  Each nondegenerate braided fusion category decomposes into a product of prime nondegenerate fusion subcategories (i.e.\!\! contains no nontrivial nondegenerate braided fusion subcategories) though this decomposition is demonstrably not unique unless the category is unpointed \cite[Theorem 4.5]{mug1}.  The classification of fusion subcategories can thus be seen as the first step toward creating a list of elementary factors from which one would hope all nondegenerate braided fusion categories can be built, as such is the case for the classification of finite simple groups.  Most ``exotic'' fusion categories (not arising in an obvious way from representation theory of finite groups/Lie algebras) are of small rank or are constructed in a way which makes a classification of fusion subcategories trivial.  Fusion subcategories of $\mathcal{C}(\mathfrak{g},k)$ are of greater complexity and were classified by Sawin \cite[Theorem 1]{Sawin06}.  The classification we present here is an extension of that for $\mathcal{C}(\mathfrak{g},k)$, though the lack of an \emph{explicit} formula for the fusion rules in $\mathcal{C}(\mathfrak{g},k)_R^0$ forces a vastly different argument to be devised which does not heavily rely on the innate geometry of $\mathfrak{g}$.  We instead argue using the properties of the dimensions of objects in $\mathcal{C}(\mathfrak{g},k)$.

\par In general, the act of passing to the category of local modules $\mathcal{C}_A^0$ over a connected \'etale algebra $A$ in a nondegenerate braided fusion category $\mathcal{C}$ preserves Witt equivalence, a concept introduced by Davydov, M\"uger, Nikshych, and Ostrik in \cite{DMNO}.  Movement between $\mathcal{C}$ and $\mathcal{C}_A^0$ is related to conformal embeddings of vertex operator algebras (e.g. \!\cite[Section 6.2]{DMNO}) and when $A$ is constructed via a Tannakian fusion subcategory of $\mathcal{C}(\mathfrak{g},k)$, $\mathcal{C}(\mathfrak{g},k)_A^0$ is referred to as a simple current extension or fixed-point theory in the parlence of conformal field theory (e.g. \!\cite[Sections 6,7]{schell}).  One also uses this construction to model anyon condensation (e.g. \cite[Chapter 6]{sebas} and \cite{kong}) in the mathematical physics literature.  It is a common belief that the Witt equivalence classes of $\mathcal{C}(\mathfrak{g},k)$ are a generating set for the unitary subgroup of the Witt group of nondegenerate braided fusion categories, but currently a proof is out of reach.  Moreover, identifying fusion subcategories and then decomposing $\mathcal{C}(\mathfrak{g},k)_A^0$ into prime factors is a crucial step in understanding the aforementioned subjects \cite[Section 5.4]{DMNO}.  The cases of $\mathfrak{sl}_2$ \cite{DNO} and $\mathfrak{sl}_3$ \cite{schopieray2017} are already complete, and Witt group relations generated by equivalence classes of these categories have been classified.  Classifying Witt group relations generated by $\mathcal{C}(\mathfrak{g},k)$ additionally relies on a classification of connected \'etale algebras in $\mathcal{C}(\mathfrak{g},k)$ which have classically been known as \emph{quantum subgroups}.  The classification of connected \'etale algebras is only complete at arbitrary levels for $\mathfrak{sl}_2$ and $\mathfrak{sl}_3$ but using a technique introduced by Ocneanu and translated by Ostrik, level bounds have been placed on the existence of ``exceptional'' connected \'etale algebras (Example \ref{ex:ade}) in $\mathcal{C}(\mathfrak{g},k)$.  These bounds were announced for generic $\mathfrak{g}$ by Gannon \cite{banffgannon} after being applied to $\mathfrak{g}_2$ and $\mathfrak{so}_5$ in \cite{schopieray2}.  Together with Corollary \ref{cor:one} at the end of Section \ref{sec:sub}, it will soon be a matter of perseverence to classify Witt group relations amongst equivalence classes generated by arbitrary $\mathcal{C}(\mathfrak{g},k)$.

\par Definitions, concepts, and prior results are summarized throughout Section \ref{sec:pre}. Theorem 1, a classification of nondegenerate fusion subcategories of $\mathcal{C}(\mathfrak{g},k)_R^0$ is proven in Section \ref{sec:sub}, and we apply this result to classify Witt group relations among the equivalence classes $[\mathcal{C}(\mathfrak{g},k)]$ where $\mathfrak{g}$ is of Type $G_2$ or $B_2$ in Section \ref{wittrel}.  There are many ways to expand the results of this paper.  Observations, questions, and open problems are presented in Section \ref{quest}.
\end{section}

\begin{section}{Preliminaries}\label{sec:pre}
\begin{subsection}{Braided fusion categories}
Our main objects of study are \emph{fusion categories} over $\mathbb{C}$, i.e.\! $\mathbb{C}$-linear, semisimple, monoidal, rigid categories with simple monoidal unit $\mathbbm{1}$ and finitely many isomorphism classes of simple objects.  Denote the set of isomorphism classes of simple objects by $\mathcal{O}(\mathcal{C})$.  We will refer the reader to \cite{tcat} for a more thorough treatment of the subject.   Monoidal, or \emph{fusion} products will be denoted by $\otimes$, appended with subscripts to differentiate multiple categories being discussed simultaneously.  Up to equivalence there is a unique fusion category with one isomorphism class of simple objects \cite[Corollary 4.4.2]{tcat} which we will denote by Vec: the category of finite dimensional $\mathbb{C}$-vector spaces.  A \emph{fusion subcategory} is a full abelian subcategory which is closed under monoidal products and quotients.  By \cite[Corollary 4.11.4]{tcat} this is sufficient to ensure the subcategory has all other necessary properties of a fusion category.  Each fusion category $\mathcal{C}$ contains $\mathrm{Vec}$ as a fusion subcategory $\otimes$-generated by $\mathbbm{1}$.  The categories Rep$(G)$, of finite dimensional representations of a finite group $G$, provide another infinite family of fusion categories.  Fusion categories may also be equipped with a \emph{braiding}: a family of natural isomorphisms $c_{X,Y}:X\otimes Y\to Y\otimes X$ for all simple $X$ and $Y$ which satisfy braid-like relations \cite[Definition 8.1.1]{tcat}.  If $c_{X,Y}$ is a braiding, then $\tilde{c}_{X,Y}:=c_{Y,X}^{-1}$ also defines a braiding on $\mathcal{C}$.  We will distinguish these braided fusion categories by denoting the latter category $\mathcal{C}^\text{rev}$.

\par Each fusion category has at least one notion of \emph{dimension} for each object $X$, $\mathrm{FPdim}(X)$, given by the Frobenius-Perron eigenvalue of the action of tensoring with $X$ \cite[Section 3.2]{tcat}.  Those fusion categories imbued with a spherical structure have a second notion of dimension \cite[Definition 4.7.11]{tcat} computed in terms of a corresponding categorical trace.  All categories considered in this exposition are \emph{pseudounitary} \cite[Section 8.4]{ENO} which implies these notions of dimension agree and allows us to denote the dimension of an object $X\in\mathcal{C}$ by $\dim(X)\in\mathbb{R}_{\geq1}$ without confusion.  Objects of dimension 1 are called \emph{simple currents} (or \emph{invertible}) and the fusion subcategory generated by invertible objects of $\mathcal{C}$ will be denoted $\mathcal{C}_\text{pt}$.  The spherical structure of $\mathcal{C}$ also induces \emph{twists} $\theta(X)\in\mathbb{C}$ for all $X\in\mathcal{O}(\mathcal{C})$ which are roots of unity for a modular tensor category $\mathcal{C}$, defined below \cite[Corollary 8.18.2]{tcat}.
\par One distinguishing property of a braided fusion category is whether it is \emph{nondegenerate}, i.e.\! $\mathcal{C}':=\{X\in\mathcal{C}:c_{Y,X}c_{X,Y}=\text{id}_{X\otimes Y}\text{ for all }Y\in\mathcal{C}\}$ is trivial.  The fusion subcategory $\mathcal{C}'$ is called the \emph{centralizer} or \emph{relative commutant} of $\mathcal{C}$ (in $\mathcal{C}$).  The following generalization and notation are borrowed from M\"uger \cite{mug1}.
\begin{definition}
If $\mathcal{D}\subset\mathcal{C}$ is a fusion subcategory, then the \emph{relative commutant of $\mathcal{D}$ inside of $\mathcal{C}$} is the full subcategory $C_\mathcal{C}(\mathcal{D})$ with objects
\begin{equation}
\{X\in\mathcal{C}:c_{Y,X}c_{X,Y}=\text{id}_{X\otimes Y}\text{ for all }Y\in\mathcal{D}\}.
\end{equation}
\end{definition}
Relative commutants $C_\mathcal{C}(\mathcal{D})$ are fusion subcategories of $\mathcal{C}$ \cite[Lemma 2.8]{mug1}.  When $\mathcal{C}'=C_\mathcal{C}(\mathcal{C})\not\simeq\text{Vec}$ we say $\mathcal{C}$ is \emph{degenerate} (or degenerately braided).  Deligne \cite[Corollary 9.9.25]{tcat} proved that if $\mathcal{C}'\simeq\mathcal{C}$ ($\mathcal{C}$ is \emph{symmetric}), then $\mathcal{C}\simeq\text{Rep}(G,\nu)$ where $G$ is a finite group and $\nu\in G$ is a central idempotent describing the braiding on $\mathcal{C}$.  If a braided fusion category of rank $n$ is spherical, one also defines the $n\times n$ \emph{$s$-matrix} consisting of traces of the double-braidings $c_{Y,X}c_{X,Y}$ for all $X,Y\in\mathcal{O}(\mathcal{C})$.  The entries of this matrix $s_{X,Y}$ can also be computed by the balancing equation \cite[Proposition 8.13.8]{tcat}
\begin{equation}\label{balancing}
s_{X,Y}=\theta(X)^{-1}\theta(Y)^{-1}\sum_{Z\in\mathcal{O}(\mathcal{C})}N_{X,Y}^Z\theta(Z)\dim(Z)
\end{equation}
where $N_{X,Y}^Z:=\dim\text{Hom}_\mathcal{C}(X\otimes Y,Z)$ are the \emph{fusion coefficients} of $\mathcal{C}$.  It is known that spherical braided fusion categories are non-degenerate if and only if the $s$-matrix is invertible \cite[Corollary 2.16]{mug1}, in which case the term \emph{modular tensor category} is used.
\end{subsection}

\begin{subsection}{Algebras in fusion categories and their modules}\label{sec:alg}
Understanding complex mathematical objects by how they act on simpler objects has been a major theme of twentieth-century mathematics and this \emph{representation theory} applied to fusion categories is no exception.  That is to say one would like to study simple (semisimple indecomposable) \emph{module categories} over a given fusion category $\mathcal{C}$ \cite[Definition 7.3.1, Section 7.5]{tcat}.  Representation theory may be studied internally to the category $\mathcal{C}$ as well.  Theorem 1 of \cite{Ostrik2003} proves that simple module categories over a fusion category $\mathcal{C}$ arise as categories of modules over certain algebra objects in $\mathcal{C}$ which we describe below.  When these algebras are commutative the categories of modules have an inherent monoidal product and provide a means to create (potentially) new fusion categories from old (often called \emph{quantum subgroups}).  Here we review the theory behind this construction referring the reader to \cite{Ostrik2003} for details.

\begin{definition}\label{def1}
An \emph{algebra} $A$ in a fusion category $\mathcal{C}$ is an object with multiplication morphism $m:A\otimes A\to A$ and unit morphism $u:\mathbbm{1}\to A$ satisying associativity and unital constraints \cite[Definition 7.8.1]{tcat}.  If $\mathcal{C}$ is braided, then $A$ is \emph{commutative} when $m\,c_{A,A}=m$.
\end{definition}

\par If $A$ is a commutative algebra in a braided fusion category $\mathcal{C}$, the category $\mathcal{C}_A$ of right $A$-modules (or left $A$-modules, or $A$-bimodules) in $\mathcal{C}$ \cite[Definition 8]{Ostrik2003} has a well-defined monoidal product $\otimes_A$.  In \cite[Theorem 1.5]{KiO}, for any $M,N\in\mathcal{C}_A$, the product $M\otimes_A N\in\mathcal{C}_A$ is defined (as an object) as a quotient of $M\otimes N\in\mathcal{C}$.  It will not be necessary in this exposition to understand $\otimes_A$ on any deeper level.  The existence of a simple monoidal unit $\mathbbm{1}_A\in\mathcal{C}_A$ requires the following assumption.

\begin{definition}
An algebra $A$ in a fusion category $\mathcal{C}$ is \emph{connected} (or \emph{haploid}) if $A$ contains $\mathbbm{1}$ as a summand with multiplicity 1.
\end{definition}

\par Thus a connected commutative algebra $A$ in a braided fusion category $\mathcal{C}$ produces a finite tensor category $\mathcal{C}_A$, and when $A$ is not connected, the resulting category of right $A$-modules $\mathcal{C}_A$ is a finite multi-tensor category.  Nonetheless any commutative algebra which is not connected canonically decomposes into a sum of connected commutative algebras, and so for most purposes it suffices to only discuss commutative algebras which are connected.

\par We require one last adjective to ensure that $\mathcal{C}_A$ is a fusion category.  As a counterexample \cite[Remark 3.7]{DMNO}, let $A$ be the group algebra of $\mathbb{Z}/2\mathbb{Z}$, a connected commutative algebra in $\mathcal{C}:=\mathrm{Vec}_{\mathbb{Z}/2\mathbb{Z}}$ as a symmetrically braided fusion category over a field \emph{of characteristic 2} (recall that the main body of this exposition takes place in characteristic zero).  Then the category of $A$-bimodules is equivalent to $\mathrm{Rep}(\mathbb{Z}/2\mathbb{Z})$ as a finite tensor category which is not semisimple by Maschke's Theorem.

\begin{definition}\label{def3}
An algebra $A$ in a fusion category $\mathcal{C}$ is \emph{separable} if $m:A\otimes A\to A$ splits as a morphism of $A$-bimodules.  Separable commutative algebras are known as \emph{\'etale algebras}.
\end{definition}

\par Proposition 2.7 of \cite{DMNO} states that $\mathcal{C}_A$ is semisimple when $A$ is separable, therefore Definitions \ref{def1}--\ref{def3} describe the desired machine for creating fusion categories $\mathcal{C}_A$ from fusion categories $\mathcal{C}$ and connected \'etale algebras $A\in\mathcal{C}$.  The category $\mathcal{C}\simeq\mathcal{C}_\mathbbm{1}$, where $\mathbbm{1}$ receives its algebra structure maps from the unit axioms of a monoidal category \cite[Definition 2.1.1]{tcat}, is a trivial application of this construction.

\begin{example}[Regular algebras]\label{ex:reg}
Consider the Tannakian fusion category $\mathrm{Rep}(G)$ for some finite group $G$.  The object $G^\ast\in\mathrm{Rep}(G)$ (complex functions on $G$) has the structure of a connected \'etale algebra \cite[Example 3.3(i)]{DMNO}.  In what follows we will be primarily considering algebras of this type (see Example \ref{ex:ade}).
\end{example}

\par To finish this subsection, we note that $\mathcal{C}_A$ is typically not braided for a braided fusion category $\mathcal{C}$ and connected \'etale algebra $A\in\mathcal{C}$.  The right $A$-modules $A\in\mathcal{C}_A$ with structure map $\rho:M\otimes A\to M$ which prevent the braiding of $\mathcal{C}$ from translating to $\mathcal{C}_A$ are those such that $\rho$ does not commute with the double-braiding $c_{A,M}c_{M,A}$ (refer to the proof of \cite[Theorem 1.10]{KiO} and the subsequent notes).

\begin{definition}
If $\mathcal{C}$ is a braided fusion category and $A\in\mathcal{C}$ a connected \'etale algebra, then $M\in\mathcal{C}_A$ with structure map $\rho:M\otimes A\to M$ is \emph{local} (or \emph{dyslectic}) if $\rho\,c_{A,M}c_{M,A}=\rho$.
\end{definition}

\par We will denote the full subcategory in $\mathcal{C}_A$ of local $A$-modules in $\mathcal{C}$ by $\mathcal{C}_A^0$, which is a braided fusion category by \cite[Theorem 1.10]{KiO}, initially due to Pareigis \cite{pareigis}.  It was proven that $\mathcal{C}_A^0$ is modular when $\mathcal{C}$ is modular in \cite[Theorem 4.5]{KiO} with the assumption that $\theta_A=\mathrm{id}_A$.  The spherical (ribbon) structure and braiding on $\mathcal{C}^0_A$ are both inherited from $\mathcal{C}$.  But the trivial twist assumption on $A$ is not strictly necessary as shown in \cite[Corollary 3.30]{DMNO}, though a trivial twist is guaranteed in the pseudounitary case \cite[Lemma 2.2.4]{schopieray2}.
\end{subsection}

\begin{subsection}{Fusion categories from Lie Theory}\label{sec:lie}

\par We will reference \cite[Section 3.3]{BaKi} as needed for an introduction to these fusion categories.  For a perspective from conformal field theory one may refer to \cite[Chapters 14--17]{conformal} and the references contained within.  For concepts from the classical theory of Lie algebras one may refer to \cite[Sections 13,21--24]{hump}.

\par Let $\mathfrak{h}$ be a Cartan subalgebra of $\mathfrak{g}$, a finite-dimensional complex simple Lie algebra, and $\langle.\,,.\rangle$ be the invariant form on $\mathfrak{h}^\ast$ normalized so $\langle\alpha,\alpha\rangle=2$ for short roots \cite[Section 5]{hump}.  It has been commonplace to study the modular tensor categories $\mathcal{C}(\mathfrak{g},k)$, consisting of highest weight integrable $\hat{\mathfrak{g}}$-modules of level $k\in\mathbb{Z}_{\geq1}$ where $\hat{\mathfrak{g}}$ is the corresponding affine Lie algebra \cite{kacpet}.  Simple objects of $\mathcal{C}(\mathfrak{g},k)$ are labelled by weights $\lambda\in\Lambda_0\subset\mathfrak{h}^\ast$, the \emph{Weyl alcove} of $\mathfrak{g}$ at level $k$.  Simple objects and their representative weights will be used interchangably but can be easily determined by context.  Equivalently \cite{fink} one considers the semisimple representation category of Lusztig's quantum group $\mathcal{U}_q(\mathfrak{g})$ when $q=e^{\pi i/m(k+h^\vee)}$ where $h^\vee$ is the dual coxeter number, and $m=1,2,3$ is a normalization dependent on $\mathfrak{g}$ \cite[Chapter 7]{BaKi}.

\par The dimension of the simple object corresponding to the weight $\lambda\in\Lambda_0$ is given by the \emph{quantum Weyl dimension formula} \cite[Equation (3.3.5)]{BaKi} which involves the quantum integers $[n]=(q^n-q^{-n})/(q-q^{-1})$ where $q$ is defined above in terms of $\mathfrak{g}$ and $k$ and will be clear from context.  The full twist on a simple object $\lambda\in\Lambda_0$ is given by $\theta(\lambda)=q^{\langle\lambda,\lambda+2\rho\rangle}$ \cite[Theorem 3.3.20]{BaKi} which is a root of unity depending on $\mathfrak{g}$, $k$, and $\lambda$, and the full modular data for $\mathcal{C}(\mathfrak{g},k)$ can be obtained from the classical Kac-Peterson formulas \cite{kacpet}.  Fusion rules for the categories $\mathcal{C}(\mathfrak{g},k)$ are also given by the \emph{quantum Racah formula} \cite[Corollary 8]{Sawin03}.  That is if $\lambda,\gamma,\mu\in\Lambda_0$, the multiplicity of $\mu$ in the product $\lambda\otimes\gamma$ is the fusion coefficient
\begin{equation}N_{\lambda,\gamma}^\mu:=\sum_{\tau\in\mathfrak{W}_0}(-1)^{\ell(\tau)}m_\lambda(\tau(\mu)-\gamma)\end{equation}
where $m_\lambda(\mu)$ is the (classical) dimension of the $\mu$-weight space of the finite dimensional irreducible representation of highest weight $\lambda$.

\begin{example}[ADE classification]\label{ex:ade}
\par In \cite{KiO} connected \'etale algebras in $\mathcal{C}(\mathfrak{sl}_2,k)$ are organized into an ``ADE classification scheme'' ($\mathrm{E}_7$ and $\mathrm{D}_n$ with $n$ odd correspond to noncommutative algebras in this classification) coined by Cappelli, Itzykson, and Zuber \cite{zuber} (see also \cite{Cappelli:2010}) in the language of modular invariant partition functions.  The connected \'etale algebra of ``Type A'' is the trivial one given by the unit object $\mathbbm{1}\in\mathcal{C}(\mathfrak{sl}_2,k)$.  Those connected \'etale algebras of ``Type D'' arise as the regular algebra (see Example \ref{ex:reg}) of Tannakian $\mathcal{C}(\mathfrak{sl}_2,k)_\text{pt}$ when $4\mid k$, and there are two remaining connected \'etale algebras of ``Type E'' in $\mathcal{C}(\mathfrak{sl}_2,10)$ and $\mathcal{C}(\mathfrak{sl}_2,28)$ which do not belong in either of the above cases, and are rightly labelled as \emph{exceptional}.  The connection between connected \'etale algebras $R\in\mathcal{C}(\mathfrak{sl}_2,k)$ and ADE Dynkin diagrams is seen by considering the fusion graph of the free module $\eta\otimes R$ where $\eta$ is the 2-dimensional fundamental representation of $\mathfrak{sl}_2$.  In the trivial case of $R=\mathbbm{1}$, the fusion graph of $\eta\otimes R\cong\eta$ recovers the ``Type A'' fusion graph for $\eta$ in $\mathcal{C}(\mathfrak{sl}_2,k)$.

\end{example}

\par We will borrow this language of the classification of connected \'etale algebras in $\mathcal{C}(\mathfrak{sl}_2,k)$ to organize connected \'etale algebras for general $\mathcal{C}(\mathfrak{g},k)$.  To this end, recall \cite[Theorem 3]{Sawin03} that $\mathcal{C}(\mathfrak{g},k)_\text{pt}$ has fusion rules matching the group multiplication of $Z_\mathfrak{g}$, the center of the simply-connected compact Lie group $G$ corresponding to $\mathfrak{g}$ \cite[Chapters II-III]{bump}.

\begin{definition}
If $H\subset Z_\mathfrak{g}$ and $\mathrm{Rep}(H)\subset\mathcal{C}(\mathfrak{g},k)_\text{pt}$ is a Tannakian fusion subcategory, we refer to the regular algebra $R\in\mathrm{Rep}(H)$ (see Example \ref{ex:reg}) as a connected \'etale algebra in $\mathcal{C}(\mathfrak{g},k)$ of \emph{Type D}.
\end{definition}

\begin{note}
When $Z_\mathfrak{g}$ is a simple group, it will be implied that $H=Z_\mathfrak{g}$.  In Type A and Type D$_n$ when $n$ is odd, we will refer to $H$ by its order as this uniquely determines $H$, and in Type D$_n$ when $n$ is even, we will refer to the three $\mathbb{Z}/2\mathbb{Z}$ subgroups of $Z_{\mathfrak{so}_n}$ as the ``diagonal'' and ``off-diagonal'' subgroups of $\mathbb{Z}/2\mathbb{Z}\times\mathbb{Z}/2\mathbb{Z}$.  When $\mathfrak{g}$ is of Type $E_8$, $F_4$, or $G_2$, then $Z_\mathfrak{g}$ is trivial and there are no nontrivial connected \'etale algebras of Type D.
\end{note}
\end{subsection}

\begin{subsection}{De-equivariantization and algebras of Type D}\label{sec:d}

\par When a fusion category $\mathcal{C}$ is equipped with $\Gamma:\text{Cat}(G)\to\text{Aut}_\otimes(\mathcal{C})$, an action of a finite group $G$ in the sense of \cite[Section 4.15]{tcat}, one may construct the category $\mathcal{C}^G$ of $G$-equivariant objects of $\mathcal{C}$.  Objects of this category are pairs $(X,\{\alpha_g\}_{g\in G})$ where $\alpha_g:\left.\Gamma\right|_g(X)\stackrel{\sim}{\to} X$ describes the $G$-equivariant structure on the object $X$. This construction produces a larger category in some sense, as a single object could have many equivariant structures.  Equivariantization retains  properties of the input category; in particular $\mathcal{C}^G$ is fusion if $\mathcal{C}$ is.  When $\mathcal{C}$ is fusion one can also make the factor of enlargement of $\mathcal{C}$ via $G$-equivariantization precise.  Proposition 4.26 of \cite{DGNO} states that $\text{FPdim}(\mathcal{C}^G)=|G|\cdot\text{FPdim}(\mathcal{C})$.

\begin{example}
For each finite group $G$ there is a unique (trivial) action on Vec.  An object of $\text{Vec}^G$ is a vector space $V$ and $\rho:G\to\text{End}(V)$ satisfying the coherence diagram of \cite[Definition 2.7.2]{tcat}, i.e.\! $\rho(gh)=\rho(g)\rho(h)$.  Hence $\text{Vec}^G\simeq\text{Rep}(G)$.  In elementary abstract algebra, $\text{FPdim}(\text{Rep}(G))=|G|\cdot\text{FPdim}(\text{Vec})=|G|$ is a familiar consequence of the theorems of Maschke and Artin-Wedderburn.
\end{example} 
\noindent For the purposes of this exposition we are considering the mutually inverse construction of equivariantization called \emph{de-equivariantization} \cite[Theorem 8.23.3]{tcat} and denoted by $\mathcal{C}_G$.  When there exists a Tannakian fusion subcategory $\mathcal{T}:=\text{Rep}(G)\subset\mathcal{C}$ the de-equivariantization $\mathcal{C}_G$ is the category of right $R$-modules in $\mathcal{C}$ where $R$ is the regular algebra of $\mathcal{T}$ \cite[Example 2.8]{DMNO}.  This motivates the seemingly ambiguous notation for categories of modules over algebras of Type D, as $\mathcal{C}_G\simeq\mathcal{C}_R$.  The act of de-equivariantization therefore does not directly transport the braiding of a fusion category $\mathcal{C}$ to $\mathcal{C}_G$: as explained in Section \ref{sec:alg}, the category of right $R$-modules for an algebra of Type D in $\mathcal{C}$ is not braided in general.  But Proposition 4.56(i) of \cite{DGNO} implies that, for Tannakian $\mathcal{T}\subset\mathcal{C}$ with regular algebra $R\in\mathcal{T}$, the category of local modules $\mathcal{C}^0_R$ is braided equivalent to the de-equivariantization of the relative commutant $(C_\mathcal{C}(\mathcal{T}))_G\subset\mathcal{C}_G$. One is now free to interpret $\mathcal{C}(\mathfrak{g},k)_R$ and $\mathcal{C}(\mathfrak{g},k)_R^0$ in either the language of local modules or that of de-equivariantization (see Examples 3.8, and 3.14 of \cite{DMNO}).

\begin{example}[de-equivariantization of $\mathcal{C}(\mathfrak{g},k)$ and the root lattice]\label{ex:deeeq}  In the case of the pseudounitary modular tensor categories $\mathcal{C}:=\mathcal{C}(\mathfrak{g},k)$ introduced in Section \ref{sec:lie} one can be precise about the relative commutant $C_\mathcal{C}(\mathcal{T})$ where $\mathcal{T}$ is a Tannakian fusion subcategory.  Corollary 6.9 of \cite{galaki} states that $C_\mathcal{C}(\mathcal{C}_\text{pt})\simeq\mathcal{C}_\text{ad}$, the \emph{adjoint} subcategory of $\mathcal{C}$, generated by $X\otimes X^\ast$ for all simple $X\in\mathcal{O}(\mathcal{C})$.  In other words, $\mathcal{C}_\mathrm{ad}$ is the trivial component for the universal grading on $\mathcal{C}$ \cite[Section 3.2]{galaki}.  The adjoint subcategory of $\mathcal{C}$ corresponds to those simple objects with corresponding dominant weights from the \emph{root lattice} of $\mathfrak{g}$, a proper sublattice of the weight lattice when $\mathcal{C}(\mathfrak{g},k)_\text{pt}\neq\mathrm{Vec}$.  In particular this implies that when $\mathcal{C}_\text{pt}$ is Tannakian with regular algebra $R$, then the category of local modules $\mathcal{C}_R^0$ is generated by the right $A$-modules with underlying objects whose corresponding weights lie in the root lattice.
\par Furthermore, if Tannakian $\mathcal{T}$ is a proper fusion subcategory of $\mathcal{C}_\text{pt}$ with regular algebra $R'$, then $C_\mathcal{C}(\mathcal{T})\supset C_\mathcal{C}(\mathcal{C}_\text{pt})\simeq\mathcal{C}_\text{ad}$.  Hence the simple objects with corresponding dominant weights in the root lattice are contained in the relative commutant of \emph{every} pointed fusion subcategory of $\mathcal{C}$, and moreover $\mathcal{C}_R^0$ is a fusion subcategory of $\mathcal{C}_{R'}^0$.
\end{example}

\end{subsection}

\begin{subsection}{Witt group of nondegenerate braided fusion categories}\label{sec:wittdef}

\par The Witt group of nondegenerate braided fusion categories $\mathcal{W}$ is an algebraic structure which organizes nondegenerate braided fusion categories $\mathcal{C}$ into \emph{Witt equivalence classes} (denoted $[\mathcal{C}]$) where, roughly speaking, the Witt equivalence class of $\mathcal{C}$ is the collection of all nondegenerate braided fusion categories which are related by the local module construction defined in Section \ref{sec:alg}.  The following definition is equivalent \cite[Corollary 5.9]{DMNO} to the original definition due to Davydov, M\"uger, Nikshych, and Ostrik given in \cite[Section 5.2]{DMNO}.

\begin{definition}
If $\mathcal{C}_1$ and $\mathcal{C}_2$ are nondegenerate braided fusion categories, then $\mathcal{C}_1$ and $\mathcal{C}_2$ are Witt equivalent if there exist connected \'etale algebras $R_1\in\mathcal{C}_1$ and $R_2\in\mathcal{C}_2$ such that $(\mathcal{C}_1)_{R_1}^0$ and $(\mathcal{C}_2)_{R_2}^0$ are braided equivalent.
\end{definition}

\par Witt equivalence classes can be characterized in terms of Drinfeld centers.  That is $[\mathcal{C}_1]=[\mathcal{C}_2]$ if and only if $\mathcal{C}_1\boxtimes(\mathcal{C}_2)^\text{rev}$ is braided equivalent to $\mathcal{Z}(\mathcal{A})$ for some fusion category $\mathcal{A}$, where $\boxtimes$ is the \emph{(Deligne) tensor product} \cite[Section 1.11]{tcat}.  In practice this definition is difficult to utilize due to the fact that little is known about the fusion categories whose centers arise due to Witt group relations, except in particular situations such as $\mathcal{Z}(\mathcal{C}_A)\simeq\mathcal{C}\boxtimes(\mathcal{C}_A^0)^\text{rev}$ \cite[Corollary 3.30]{DMNO} when $A$ is a connected \'etale algebra in a nondegenerate braided fusion category $\mathcal{C}$ (i.e. \!$[\mathcal{C}]=[\mathcal{C}_A^0]$).

\par The set of Witt equivalence classes is a group with commutative associative operation $[\mathcal{C}][\mathcal{D}]:=[\mathcal{C}\boxtimes\mathcal{D}]$ and inverses $[\mathcal{C}]^{-1}=[\mathcal{C}^\text{rev}]$.  The structure of the Witt group is known in broad strokes.  It is a 2-primary abelian group containing the subgroups $\mathcal{W}_\text{pt}$, consisting of Witt equivalence classes of pointed modular tensor categories, which is contained in $\mathcal{W}_\text{un}$, consisting of Witt equivalence classes of unitary modular tensor categories.  All relations in $\mathcal{W}_\text{pt}$ are known due to its relation to the classical Witt group of finite abelian groups and nondegenerate quadratic forms \cite[Section 5.3]{DMNO}.  Some nonpointed relations in $\mathcal{W}_\text{un}$ are known due to conformal embeddings of vertex operator algebras \cite{bb,kw,sw}.  Each conformal embedding $(\mathfrak{g}_1)_{k_1}\subset(\mathfrak{g}_2)_{k_2}$ of complex finite dimensional simple Lie algebras (where the notation implies $g_2$-modules of level $k_2$ restrict to a $\mathfrak{g}_1$-module of level $k_1$) implies $[\mathcal{C}(\mathfrak{g}_1,k_1)]=[\mathcal{C}(\mathfrak{g}_2,k_2)]$.  This construction extends to embeddings of products of simple Lie algebras as well.  For complete details, refer to Section 6.2 of \cite{DMNO} and references within as we will use this fact in both the proof of our main theorem and in its application to Witt group relations.  Only those relations between Witt equivalence classes of $\mathcal{C}(\mathfrak{sl}_2,k)$ and $\mathcal{C}(\mathfrak{sl}_3,k)$ have been completely classified (this contains some subset of the pointed classification).

\par To classify relations in $\mathcal{W}$, one studies the closely-related super Witt group $s\mathcal{W}$ of slightly degenerate braided fusion categories.  A braided fusion category $\mathcal{C}$ is slightly degenerate if $\mathcal{C}'\simeq\text{sVec}$, the category of super vector spaces \cite[Example 8.3.6]{tcat}.  The group operation in $s\mathcal{W}$ is the Deligne tensor product relative to $\text{sVec}$ \cite[Section 4.1]{DMNO} and its structure is known explicitly \cite[Proposition 5.18]{DNO}:
\begin{equation}
s\mathcal{W}\simeq s\mathcal{W}_\text{pt}\bigoplus s\mathcal{W}_2\bigoplus s\mathcal{W}_\infty,
\end{equation}
where $s\mathcal{W}_\text{pt}$ is the subgroup of slightly-degenerate pointed premodular categories, $s\mathcal{W}_2$ is an elementary abelian 2-group, and $s\mathcal{W}_\infty$ is a free abelian group of countable rank.  There is a group homomorphism $\Sigma:\mathcal{W}\to s\mathcal{W}$ such that $[\mathcal{C}]\mapsto[\mathcal{C}\boxtimes\text{sVec}]$ whose kernel is generated by the Witt equivalence classes of the Ising categories \cite[Appendix B]{DGNO}.
\par Remark 5.11 of \cite{DNO} summarizes how Witt group relations in $\mathcal{W}$ can be classified: there are no non-trivial relations in $s\mathcal{W}$ except of the form $[\mathcal{C}]=[\mathcal{C}]^{-1}$ for some slightly nondegenerate braided fusion category $\mathcal{C}$ which (i) is not pointed (ii) has no fusion subcategories containing $\text{sVec}$ aside from $\text{sVec}$ and $\mathcal{C}$, and (iii) contains no nontrivial connected \'etale algebras.  The third condition on braided fusion categories is often called \emph{completely anisotropic}.  Hence if one can identify simple, completely anisotropic representatives of Witt equivalence classes in $\mathcal{W}$ which are not pointed, and demonstrate that their image under the homomorphism $\Sigma$ satisfies properties (i)--(iii) above then the only relations amongst these representative categories are their Witt classes being order 2.

\begin{example}\label{ex:g}
Consider the categories $\mathcal{C}(\mathfrak{g}_2,1)$ and $\mathcal{C}(\mathfrak{g}_2,2)$ of ranks 2 and 4, respectively.  These categories are are not pointed using the numerical data found in \cite[Section 2.3.4]{schopieray2}, they are simple by \cite[Theorem 1]{Sawin06}, and are completely anisotropic by \cite[Lemma 2.2.3]{schopieray2} as they contain no nontrivial objects with trivial twist.  As such their images under the homomorphism $\Sigma$ are not pointed and have no proper fusion categories containing $\text{sVec}$ besides $\text{sVec}$ itself from the fact that $\mathcal{C}(\mathfrak{g}_2,1)$ and $\mathcal{C}(\mathfrak{g}_2,2)$ were unpointed.  Lastly we note that the \emph{multiplicative central charges} \cite[Section 8.15]{tcat} of these categories are not $\pm1$.  As central charge is an invariant of (pseudounitary) Witt equivalence \cite[Lemma 5.27]{DMNO} which is multiplicative across Deligne tensor products, this implies the order of the equivalency classes of these categories cannot be 2.  Moreover $[\mathcal{C}(\mathfrak{g}_2,1)]$ and $[\mathcal{C}(\mathfrak{g}_2,2)]$ generate a subgroup of $\mathcal{W}$ isomorphic to $\mathbb{Z}^2$.
\end{example}
%
\end{subsection}

\end{section}


\begin{section}{Fusion subcategories of $\mathcal{C}(\mathfrak{g},k)_R^0$}\label{sec:sub}

\par A fusion category (braided or not) is \emph{simple} if it contains no nontrivial proper fusion subcategories; a braided fusion category is \emph{prime} if it contains no nontrivial \emph{nondegenerate} fusion subcategories.  The concept and title of prime nondegenerate braided fusion categories is justified because each nondegenerate braided fusion category factors into a $\boxtimes$-product of primes \cite[Proposition 2.2]{DMNO} (using \cite[Theorem 3.13]{DGNO} and \cite[Theorems 4.2, 4.5]{mug1}).   Braided fusion categories $\mathcal{C}$ always contain the fusion subcategory $\mathcal{C}'\subset\mathcal{C}$ of degenerate (or transparent) objects (see Section \ref{sec:pre}).  Fusion subcategories under the name of closed subsets of $\mathcal{C}(\mathfrak{g},k)$ were classified by Sawin \cite[Theorem 1]{Sawin06}.  The fusion rules for the de-equivariantizations $\mathcal{C}(\mathfrak{g},k)_R^0$ are less tractable, not being described explicitly (or geometrically) by the quantum Racah formula as was used in \cite{Sawin06}, but a quotient of the fusion rules of $\mathcal{C}(\mathfrak{g},k)$ in \cite[Theorem 1.5]{KiO}.  These fusion rules could also be accessible from explicit computation of the $s$-matrices of $\mathcal{C}(\mathfrak{g},k)_R^0$ as in \cite{baver1996fusion,fuchs1996matrix} followed by application of the Verlinde formula, but proving properties in generality from this method seems unlikely.  The main result of this section (Theorem \ref{th:simplicity}) is a classification of nondegenerate fusion subcategories of $\mathcal{C}(\mathfrak{g},k)_R^0$ for arbitrary nontrivial connected \'etale algebras $R$ of Type D (Section \ref{sec:alg}), and restrictive conditions on when exceptional degenerate fusion categories exist.  Sawin's classification implies the case when $R$ is the algebra $\mathbbm{1}$ and thus $\mathcal{C}(\mathfrak{g},k)_\mathbbm{1}^0\simeq\mathcal{C}(\mathfrak{g},k)$, and the arguments below rely heavily on \cite{Sawin06}. 

\begin{example}\label{ex:su42}
The category $\mathcal{C}(\mathfrak{sl}_4,4)_\text{pt}$ has $\mathbb{Z}/4\mathbb{Z}$ fusion.  We will use the labelling conventions for weights as in \cite[Section 3.1]{gannon2002}.  There is a subgroup $H:=\mathbb{Z}/2\mathbb{Z}\subset \mathbb{Z}/4\mathbb{Z}=Z_{\mathfrak{sl}_4}$ whose simple objects correspond to the weights $\mathbbm{1}$ and $4\lambda_2$.  In the notation of \cite{Sawin06}, Sawin's classification states the fusion subcategories of $\mathcal{C}(\mathfrak{sl}_4,4)$ are the three pointed fusion subcategories, and their relative commutants, respectively:
\begin{equation}
\text{Vec}\subset\Delta_H\subset\mathcal{C}(\mathfrak{sl}_4,4)_\text{pt}\qquad\text{and}\qquad\mathcal{C}(\mathfrak{sl}_4,4)\supset
\Gamma_H\supset\mathcal{C}(\mathfrak{sl}_4,4)_\text{ad}.
\end{equation}
The simple current $4\lambda_2$ has trivial twist, so there exists a unique structure of a nontrivial connected \'etale algebra of Type D on $R:=\mathbbm{1}\oplus4\lambda_2$.  One would expect there to be four fusion subcategories of $\mathcal{C}(\mathfrak{sl}_4,2)_R^0$ corresponding to the fusion subcategories of $(\mathcal{C}(\mathfrak{sl}_4,2)_R^0)_\mathrm{pt}$, of rank $|Z_\mathfrak{sl_4}/H|=2$, and their relative commutants.  Using \cite[Lemma 3.4]{KiO} and \cite[Lemma 4]{Ostrik2003}, one computes $\mathcal{C}(\mathfrak{sl}_4,4)_R^0$ is rank 14 and $(\mathcal{C}(\mathfrak{sl}_4,4)_R^0)_\mathrm{pt}\simeq\mathrm{sVec}$ whose relative commutant is $(\mathcal{C}(\mathfrak{sl}_4,4)_R^0)_\mathrm{ad}$, a rank 7 fusion subcategory by \cite[Corollary 6.9]{galaki}.  But $(\mathcal{C}(\mathfrak{sl}_4,4)_R^0)_\mathrm{ad}$ is degenerate by the double-centralizer property \cite[Theorem 3.2(i)]{mug1}, and contains two unexpected degenerate fusion subcategories of rank 4, who are the relative commutant of one another.  The additional degenerate fusion subcategories can only occur in Type A and D by Theorem \ref{th:simplicity} below, and the general classification continues the pattern illustrated here.
\end{example}

\begin{theorem}\label{th:simplicity}
Let $R$ be the regular algebra of a Tannakian fusion subcategory $\rep(H)\subset\mathcal{C}(\mathfrak{g},k)_\text{pt}$ for a complex finite dimensional simple Lie algebra $\mathfrak{g}$, level $k\in\mathbb{Z}_{\geq1}$, and nontrivial $H\subset Z_\mathfrak{g}$.  A fusion subcategory of $\mathcal{C}(\mathfrak{g},k)_R^0$ is either a pointed fusion subcategory corresponding to a subgroup of $\mathcal{O}((\mathcal{C}(\mathfrak{g},k)_R^0)_\mathrm{pt})\subset Z_\mathfrak{g}/H$, the relative commutant of a pointed fusion subcategory, or one of the following exceptions:
\begin{itemize}
\item[(a)] fusion subcategories of $\mathcal{C}(\mathfrak{sl}_3,3)_R^0$, each generated by a nontrivial invertible object,
\item[(b)] fusion subcategories of $\mathcal{C}(\mathfrak{so}_{2n+1},2)_R^0$ in one-to-one correspondence with proper nontrivial subgroups of $\mathbb{Z}/(2n+1)\mathbb{Z}$,
\item[(c)] fusion subcategories of $\mathcal{C}(\mathfrak{so}_{2n+1},4)_R^0$ which are equivalent (as fusion categories) to $\mathcal{C}(\mathfrak{sl}_2,2n+1)_\text{ad}$,
\item[(d)] fusion subcategories of $\mathcal{C}(\mathfrak{sp}_8,2)_R^0$ which are equivalent (as fusion categories) to $\mathcal{C}(\mathfrak{sl}_2,5)_\text{ad}$,
\item[(e)] with $n\in\mathbb{Z}_{\geq4}$ and $H$ the diagonal $\mathbb{Z}/2\mathbb{Z}\subset\mathbb{Z}/2\mathbb{Z}\times\mathbb{Z}/2\mathbb{Z}=Z_{\mathfrak{so}_{2n}}$, the fusion subcategories of $\mathcal{C}(\mathfrak{so}_{2n},2)_R^0$ in one-to-one correspondence with nontrivial proper subgroups of $\mathbb{Z}/2n\mathbb{Z}$,
\item[(f)] with $n\equiv0\pmod{4}$ and $H=Z_{\mathfrak{so}_{2n}}$, the fusion subcategories of $\mathcal{C}(\mathfrak{so}_{2n},2)_R^0$ in one-to-one correspondence with nontrivial proper subgroups of $\mathbb{Z}/(n/2)\mathbb{Z}$,
\item[(g)] with $n\equiv0\pmod{4}$ where $H$ is an off-diagonal $\mathbb{Z}/2\mathbb{Z}\subset\mathbb{Z}/2\mathbb{Z}\times\mathbb{Z}/2\mathbb{Z}$, the fusion subcategories of $\mathcal{C}(\mathfrak{so}_{2n},2)_R^0$ in one-to-one correspondence with nontrivial proper fusion subcategories of $\mathcal{C}(\mathfrak{so}_{n/2},2)$,
\item[(h)] with $n\in\mathbb{Z}_{\geq4}$ even and $H=Z_{\mathfrak{so}_{2n}}$, the fusion subcategories of $\mathcal{C}(\mathfrak{so}_{2n},4)_R^0$ equivalent (as fusion categories) to $\mathcal{C}(\mathfrak{sl}_2,2n)_S^0$ where $S$ is the unique nontrivial connected \'etale algebra of Type D in $\mathcal{C}(\mathfrak{sl}_2,2n)$, and
\item[(i)] any exceptional degenerate braided fusion subcategories of relative commutants of degenerate pointed fusion subcategories in Types A and D when $H\subsetneq Z_\mathfrak{g}$.
\end{itemize}
\end{theorem}

\begin{proof} Most exceptions occur at level $k=2$ so our first results will assume $k\geq3$.  We begin by analyzing the pointed fusion subcategory of $\mathcal{C}(\mathfrak{g},k)_R^0$.

\begin{lemma}\label{lem:point} If $k\geq3$, then $\mathcal{C}(\mathfrak{g},k)_R^0$ contains no objects of integer dimension outside of $(\mathcal{C}(\mathfrak{g},k)_R^0)_\mathrm{pt}$ except possibly in the cases $\mathcal{C}(\mathfrak{sl}_2,4)$ and $\mathcal{C}(\mathfrak{sl}_3,3)$.
\end{lemma}
\begin{proof}
Note that objects of integer dimension in $\mathcal{C}(\mathfrak{g},k)$ have been completely classified, as described in \cite{banffschopieray}.  For $k\geq3$, the only objects of integer dimensions greater than 1 occur in the two listed exceptions.  Denote the dimension of $X\in\mathcal{C}(\mathfrak{g},k)_R$ by $\dim_R(X)$.  If $M\in\mathcal{C}(\mathfrak{g},k)_R^0$ with $\dim_R(M)\in\mathbb{Z}$, then $|H|\dim_R(M)=\dim(M)$ \cite[Theorem 1.18]{KiO}.  Thus $\dim(M)\in\mathbb{Z}$ and $M\in\mathcal{C}(\mathfrak{g},k)_\mathrm{pt}$ by \cite{banffschopieray} aside from the exceptions.  Hence as an $R$-module, $M\in((\mathcal{C}(\mathfrak{g},k)_R^0)_\mathrm{pt}$.
\end{proof}

\par The categories $\mathcal{C}(\mathfrak{sl}_2,4)$ and $\mathcal{C}(\mathfrak{sl}_3,3)$ have unique nontrivial connected \'etale algebras $R$ of Type D.  The category $\mathcal{C}(\mathfrak{sl}_2,4)_R^0$ is pointed of rank 3 which is realizable from the conformal embedding $(\mathfrak{su}_2)_4\subset(\mathfrak{sl}_3)_1$ \cite[Section 6.4(4)]{DMNO}.  This category has only 2 fusion subcategories and so it is not exceptional.  The category $\mathcal{C}(\mathfrak{sl}_3,3)_R^0$ is pointed of rank 4 with $\mathbb{Z}/2\mathbb{Z}\times\mathbb{Z}/2\mathbb{Z}$ fusion \cite[Example 3.5.1]{schopieray2017}.  Lemma \ref{lem:point} implies that invertible $R$-modules are in one-to-one correspondence with free modules on the invertible objects of $\mathcal{C}(\mathfrak{g},k)$.  These free modules are simple because $\mathcal{C}(\mathfrak{g},k)_\mathrm{pt}$ has fusion rules of the abelian group $Z_\mathfrak{g}$ (see Equation \ref{ostlemma} below).  Orbits of $\mathrm{Rep}(H)$ acting on $Z_\mathfrak{g}$ are the cosets $Z_\mathfrak{g}/H$, proving the first first major claim of our theorem as simple invertible local modules form a subgroup of $Z_\mathfrak{g}/H$.  If $H=Z_\mathfrak{g}$, then $(\mathcal{C}(\mathfrak{g},k)_R^0)_\mathrm{pt}\simeq\mathrm{Vec}$, and the following lemma proves Theorem \ref{th:simplicity}(i) since $Z_\mathfrak{g}$ is a simple group for all $\mathfrak{g}$ except Types A and D.

\begin{lemma}
If $k\geq3$, all degenerate fusion subcategories of $\mathcal{C}(\mathfrak{g},k)_R^0$ are contained in relative commutants of degenerate pointed fusion subcategories.
\end{lemma}

\begin{proof}
A degenerate fusion subcategory $\mathcal{D}\subset\mathcal{C}(\mathfrak{g},k)_R^0$ has a nontrivial symmetric fusion subcategory: its centralizer $\mathcal{D}'$.  Deligne's Theorem implies $\mathcal{D}'$ is integral and by Lemma \ref{lem:point}, $\mathcal{D}'\subset(\mathcal{C}(\mathfrak{g},k)_R^0)_\mathrm{pt}$, where the two exceptions are irrelevant as they are pointed.  Moreover $\mathcal{D}'\subset C_\mathcal{C}(\mathcal{D})$, hence $C_\mathcal{C}(\mathcal{D}')\supset C_\mathcal{C}(C_\mathcal{C}(\mathcal{D}))=\mathcal{D}$.
\end{proof}

Now the main body of the proof is for $k\in\mathbb{Z}_{\geq3}$, where the obstruction to $\mathcal{C}(\mathfrak{g},k)_R^0$ having exceptional (i.e. \!neither pointed nor the relative commutant of a pointed subcategory) nondegenerate fusion subcategories is the existence of a \emph{simple} free local $R$-module $\overline{\beta}:=\beta\otimes R$ for some $\beta\in\mathcal{C}(\mathfrak{g},k)$ of small dimension.  The simple object $\overline{\beta}$ must factor as a product of other simple $R$-modules which is rarely possible if $\dim(\beta)$ is small.  Subsections \ref{simpleA}--\ref{secE} describe this $\overline{\beta}$ for each Lie algebra $\mathfrak{g}$ and discuss the cases $k=1,2$ and any exceptional cases separately.  When $R\in\mathcal{C}_\mathrm{pt}$ is a regular algebra, the simplicity of $\overline{\lambda}$ can be quickly determined by Lemma 4 of \cite{Ostrik2003} which states
\begin{equation}\label{ostlemma}
\dim\mathrm{Hom}_{\mathcal{C}_R}(\lambda\otimes R,\lambda\otimes R)=\dim\mathrm{Hom}_\mathcal{C}(\lambda,\lambda\otimes R).
\end{equation}
Therefore $\overline{\lambda}=\lambda\otimes R$ is simple if $\lambda\otimes\delta\not\cong\lambda$ for any nontrivial summand $\delta\subset R$.  We will use this fact repeatedly below.  In the case of $\lambda\in\mathcal{C}(\mathfrak{g},k)$ this is a trivial verification as the action of the simple currents corresponds to symmetries of the affine Dynkin diagram of $\mathfrak{g}$.

\par If $\mathcal{C}(\mathfrak{g},k)_R^0$ contains a \emph{nondegenerate} fusion subcategory $\mathcal{D}$, denote the relative commutant of $\mathcal{D}$ in $\mathcal{C}(\mathfrak{g},k)_R^0$ by $\mathcal{E}$ and thus $\mathcal{C}(\mathfrak{g},k)_R^0\simeq\mathcal{D}\boxtimes\mathcal{E}$ \cite[Theorem 4.2]{mug1}.  We will assume $\mathcal{D}$ and $\mathcal{E}$ are not exceptional, that is each contains at least one noninvertible simple object.  When $k\geq3$ we claim any free local $R$-module $\overline{\lambda}:=\lambda\otimes R$, where $\lambda$ is not a \emph{corner} of $\Lambda_0$ (i.e. \!$\lambda=k\lambda_j$ for a fundamental weight $\lambda_j$) $\otimes_R$-generates the fusion subcategory $\mathcal{L}\subset\mathcal{C}(\mathfrak{g},k)_R^0$ generated by free  $R$-modules on simple objects corresponding to weights in the root lattice, $\mathcal{C}(\mathfrak{g},k)_\mathrm{ad}$.  To see this, the explanation at the end of \cite[Section 2]{Sawin06} describes how if $\lambda\in\Lambda_0$ is not a corner \cite[Lemma 3]{sawin2002}, then there exists $n\in\mathbb{Z}_{\geq1}$ such that  $\theta$ is a subobject of $(\lambda\otimes\lambda^\ast)^{\otimes n}$.  The proof of Lemma 3 of \cite{Sawin06} implies $\theta$ generates the root lattice, and moreover $\overline{\lambda}$ generates $\mathcal{L}$ as the free module functor is monoidal, i.e.
\begin{equation}
(\overline{\lambda\otimes\lambda^\ast})^{\otimes_R n}\cong\overline{(\lambda\otimes\lambda^\ast)^{\otimes n}}\supset\overline{\theta}.
\end{equation}
But as $C_{\mathcal{C}(\mathfrak{g},k)}(\mathcal{C}(\mathfrak{g},k)_\mathrm{ad})\simeq\mathcal{C}(\mathfrak{g},k)_\mathrm{pt}$, we conclude $C_{\mathcal{C}(\mathfrak{g},k)}(\mathcal{D})$ is pointed (see Example \ref{ex:deeeq}) and $\mathcal{D}$ is not exceptional.  For $k\geq3$, no exceptional fusion categories exist in Sawin's classification which means the above reasoning applies even when $\lambda$ is a non-invertible corner.  That is if $\lambda\in\mathcal{D}$ is a simple current (invertible corner) then $\mathcal{D}$ is pointed, otherwise $\mathcal{E}$ is pointed.  Moreover, exceptional nondegenerate fusion subcategories may only occur when $\mathcal{D}$ and $\mathcal{E}$ consist solely of non-free $R$-modules.  This is a restrictive constraint.  Indeed, assume a non-invertible free simple (by assumption) $R$-module $\overline{\beta}$ factors exceptionally in $\mathcal{C}(\mathfrak{g},k)_R^0$ as
\begin{equation}
\overline{\beta}\cong\tilde{\lambda}\otimes_R\tilde{\mu}\label{qa}
\end{equation}
where the notation $\tilde{\lambda}\in\mathcal{D},\tilde{\mu}\in\mathcal{E}$ is used to imply these are simple summands of $\overline{\lambda},\overline{\mu}\in\mathcal{C}(\mathfrak{g},k)^0_A$.   By the above discussion, neither $\tilde{\lambda}$ nor $\tilde{\mu}$ are invertible.  Without loss of generality assume $\dim_R(\tilde{\lambda})\leq\dim_R(\tilde{\mu})$. From \cite[Theorem 1.18]{KiO} we have
\begin{equation}\label{newcrux}
\dim(\beta)\geq\dfrac{\dim(\lambda)^2}{|Z_\mathfrak{g}|^2}
\end{equation}
for some non-invertible simple $\lambda\in\mathcal{C}(\mathfrak{g},k)$.  This inequality will be sufficient for our argument.


\begin{subsection}{Type $A_{n-1}$ for $n\geq2$}\label{simpleA}

\par We assume $n\geq4$ as the cases $n=2,3$ are complete \cite[Section 5.5]{DNO}\cite[Theorem 1]{schopieray2017}.  The categories $\mathcal{C}(\mathfrak{sl}_n,1)$ are pointed so all fusion subcategories of $\mathcal{C}(\mathfrak{sl}_n,1)_R^0$ are not exceptional.  Sawin's classification identifies no exceptional fusion subcategories of $\mathcal{C}(\mathfrak{sl}_n,2)$ and Lemma \ref{lem:point} is true for $\mathcal{C}(\mathfrak{sl}_n,2)$ unless $n=4$.  So we may extend the general proof to $k=2$ for Type A.  The category $\mathcal{C}(\mathfrak{sl}_4,2)$ has a unique nontrivial connected \'etale algebra $R$ and $\mathcal{C}(\mathfrak{sl}_4,2)_R^0$ is pointed with $\mathbb{Z}/3\mathbb{Z}$ fusion.  By computing relative commutants, this case is not exceptional.  For $k\geq2$ the dominant root is $\beta:=\lambda_1+\lambda_n$.  A generating simple current acts by $\beta\mapsto(k-2)\lambda_1+\lambda_2$, thus $\overline{\beta}$ is simple when $n\geq4$ except $n=4$ with $k=2$ which we have described.  One computes
\begin{equation}
\lambda_j\otimes\lambda_j^\ast=\bigoplus_{\ell=0}^j(\lambda_m+\lambda_{k-m}).
\end{equation}
Therefore $\dim(\lambda_j)\leq\dim(\lambda_{j+1})$ for all $0\leq j<n/2$.  These are the simple noninvertible objects of smallest dimension (along with their orbits under the simple currents and charge conjugation) in $\mathcal{C}(\mathfrak{sl}_n,k)$ \cite[Lemma 2(1)]{gannon1996}.   We claim even these weights are too large to be factors of $\overline{\beta}$.  Note that $\dim(\beta)=[n]^2-1<[n]^2$, and so if we can show $\dim(\lambda_{n/2})/[n]>n$, then $\overline{\beta}$ cannot factor via fundamental weights, and moreover by any noninvertible weights since they are of necessarily larger dimension.  Furthermore $\beta$ is in the root lattice, so rank-level duality \cite[Theorem 5.1]{ranklevel} implies we need only check for factorization of $\overline{\beta}$ in terms of dimensions when $k\geq n$.  We will make the following argument for $n$ even; the argument when $n$ is odd is similar.  The Weyl dimension formula gives
\begin{equation}\label{aineq}
\dfrac{\dim(\lambda_{n/2})}{[n]}=\dfrac{[n-1][n-2]\cdots[n/2+1]}{[n/2][n/2-1]\cdots[1]}\geq\dfrac{(n-1)(n-2)\cdots(n/2+1)}{2^{n/2-1}(n/2)(n/2-1)\cdots(2)},
\end{equation}
where the inequality is due to $[m]\geq m/2$ whenever $k\geq n$ and $[m]\leq m$ in generality \cite[Section 3.2]{schopieray2}.  The right-hand side of (\ref{aineq}) is strictly greater than $n$ for $n=18$ and the derivative is strictly greater than 1 as a function of $n$ (the bound is independent of $k\geq n$).  We then compute $\dim(\lambda_{n/2})/[n]>n$ for $17\geq n\geq9$ directly when $n=k$, and note that $\dim(\lambda_{n/2})/[n]-n$ is an increasing function of $k$ for fixed $n$.  It remains to analyze the cases $n=4,5,6,7,8$ separately where the initial inequality fails.

\par For the prime cases $n=5,7$ with $k\geq n$ we note that there exists a nontrivial Tannakian fusion subcategory precisely when $k=nm$ for some $m\in\mathbb{Z}_{\geq 1}$ in which case there is exactly one object, $\lambda:=(k/n)(\lambda_1+\cdots+\lambda_n)$ such that $\overline{\lambda}$ is not simple ($\lambda$ is fixed by the simple currents).  One can then verify that $\dim(\lambda)/n-n>0$ for $m=1$, and is an increasing function of $m$ (for fixed $n$).  Thus $\overline{\beta}$ cannot factor into simple summands of $\overline{\lambda}$ in these cases.

\par For the case $n=8$ we compute $\dim(\lambda_4)/[8]\approx5.01$ when $k=8$ and is an increasing function of $k$.  Hence $\overline{\beta}$ cannot factor when $|H|=2,4$.  Similarly $\dim(\lambda_4)/[8]\approx8.008>8$ when $k=31$ and is an increasing function of $k$.  Hence $\overline{\beta}$ could only factor at level $k=16$.  Again there is a unique weight $\lambda$ stabilized by the simple currents as in the $n=5,7$ cases and we verify $\dim(\lambda)^2/8^2>\dim(\beta)$.

\par When $n=6$, we compute $\dim(\lambda_2)/[6]\approx2.366$ when $k=6$, and $\dim(\lambda_2)/[6]\approx3.025$ when $k=16$ and both are increasing functions of $k$.  Hence $\overline{\beta}$ cannot factor when $|H|=2$ and $k$ is arbitrary, cannot factor when $|H|=3$ unless $k=6,9,12,15$, and cannot factor when $|H|=6$ unless $k=12$.  The $k=12$ case is eliminated by the dimension check in the $n=8$ case.  The $|H|=3$ cases can be eliminated by dimension checks as well.  For instance when $k=6$, there are only 3 weights $\lambda$ such that $\overline{\lambda}$ is not simple and $\dim(\lambda)^2/3^2>\dim(\beta)$ for each.

\par When $n=4$ the above arguments fail completely as $\dim(\lambda_1)/[4]<2$ for all $k\in\mathbb{Z}_{\geq1}$.  For $k\not\equiv0\pmod{8}$, $|H|=2$ when $k$ is even.  The weight $\lambda_2$ is centralized by $\mathcal{C}(\mathfrak{sl}_4,k)_\mathrm{pt}$ and hence $\overline{\lambda_2}$ must factor.  But this would imply $\dim(\lambda_2)^2/4\leq\dim(\lambda_2)$ and thus $\dim(\lambda_2)\leq4$ which is true only when $k=2,4$ (see Example \ref{ex:su42} and the explanation at the start of this section).  Finally, when $k=8m$ for some $m\in\mathbb{Z}_{\geq1}$, $|H|=4$ and there is a unique weight fixed by the simple currents: $\lambda:=(k/8)(\lambda_1+\lambda_2+\lambda_3)$.  The $R$-modules with 2 simple summands are those weights fixed by the unique simple current of order two, hence $m_1\lambda_1+m_2\lambda_2+m_3\lambda_3$ such that $m_1=m_3$ and $k-m_1-m_3=2m_2$.  Therefore the total number of simple $R$-modules which are not free is $2m+4$.   The total number of simple objects of $\mathcal{C}(\mathfrak{sl}_4,8m)$ is $\binom{8m+4}{3}$, so the rank of $\mathcal{C}(\mathfrak{sl}_4,8m)_R^0$ is greater than or equal to $(1/4)\binom{8m+4}{3}=(1/24)(8m+2)(8m+3)(8m+4)$.  But if $\mathcal{C}(\mathfrak{sl}_4,8m)_R^0$ factored in an exceptional way, each factor consists of simple $R$-modules which are not free, hence the rank is at most $(1/4)(2m+4)^2$.  This implies $\mathcal{C}(\mathfrak{sl}_4,8m)_R^0$ cannot factor exceptionally since for all $m\in\mathbb{Z}_{\geq1}$,
$(1/24)(8m+2)(8m+3)(8m+4)>(1/4)(2m+4)^2$.

\end{subsection}


\begin{subsection}{Type $B_n$ for $n\in\mathbb{Z}_{\geq2}$}\label{secB}

\par The categories $\mathcal{C}(\mathfrak{so}_{2n+1},1)$ are Ising categories \cite[Appendix B.5]{DGNO}.  The categories $\mathcal{C}(\mathfrak{so}_{2n+1},2)$ are \emph{(odd) metaplectic} \cite[Definition 3.1]{metaplectic} and have a unique nontrivial connected \'etale algebra $R$ of Type D and $\mathcal{C}(\mathfrak{so}_{2n+1},2)_R^0$ is pointed with $\mathbb{Z}/(2n+1)\mathbb{Z}$ fusion \cite[Lemma 3.4]{metaplectic}.  For $n\in\mathbb{Z}_{\geq3}$, $\mathcal{C}(\mathfrak{so}_{2n+1},k)_\text{pt}$ has $\mathbb{Z}/2\mathbb{Z}$ fusion which is Tannakian only if $k$ is even.  The nontrivial simple current acts on the short dominant root $\beta:=\lambda_1$ by $\beta\mapsto(k-1)\lambda_1$ \cite[Section 3.2]{gannon2002} and thus $\overline{\beta}$ is simple.  The weight $\beta$ is the simple object with least dimension greater than 1 \cite[Proposition A]{gannon1996} along with $(k-1)\lambda_1$.  Hence if $\overline{\beta}$ factors, inequality (\ref{newcrux}) implies
\begin{equation}\label{barg}
4\dim(\beta)\geq\dim(\lambda)^2\geq\dim(\beta)^2.
\end{equation}
In other words, $\dim(\beta)\leq4$.  As $\dim(\beta)=[2n]+[1]$ is both an increasing function of $n\in\mathbb{Z}_{\geq2}$ and $k\in\mathbb{Z}_{\geq4}$, we have only to check $\dim(\beta)>4$ in $\mathcal{C}(\mathfrak{so}_7,5)$ to imply there are no exceptional nondegenerate fusion subcategories for $n\geq2$ while $k\geq5$.  The infinite family of cases $k=4$ produce the exceptions noted in Theorem \ref{th:simplicity} (c) which are justified by the following lemma.

\begin{lemma}\label{lem:B} If $n\in\mathbb{Z}_{\geq2}$ and $R$ is the unique non-trivial connected \'etale algebra of Type D in $\mathcal{C}(\mathfrak{so}_{2n+1},4)$, then $(\mathcal{C}(\mathfrak{so}_{2n+1},4)_R^0)^\mathrm{rev}\simeq\mathcal{C}(\mathfrak{sl}_2,2n+1)_\mathrm{ad}^{\boxtimes2}$ is a braided equivalence.
\end{lemma}
\begin{proof}
This follows from the rank-level embedding $\mathfrak{so}(2n+1)_4\times\mathfrak{su}(2)_{2n+1}\subseteq\mathfrak{sp}(2(2n+1))_1$ \cite[Appendix]{DMNO} and the isomorphism of Grothendieck rings of $\mathcal{C}(\mathfrak{sl}_2,2n+1)$ and $\mathcal{C}(\mathfrak{sp}_{2(2n+1)},1)$ \cite[Theorem 5.1]{gannon2002}.  Although $\mathcal{C}(\mathfrak{sl}_2,2n+1)$ and $\mathcal{C}(\mathfrak{sp}_{2(2n+1)},1)$ have the same fusion rules, they are not equivalent as modular tensor categories.  By computing twists, one verifies that
\begin{equation}
\mathcal{C}(\mathfrak{sp}_{2(2n+1)},1)\simeq\mathcal{C}(\mathfrak{sl}_2,2n+1)_\mathrm{ad}^\mathrm{rev}\boxtimes\mathcal{C}(\mathfrak{sl}_2,2n+1)_\mathrm{pt}.
\end{equation}

In particular the rank-level embedding with \cite[Proposition 5.4]{DMNO} and the $\boxtimes$-decomposition of $\mathcal{C}(\mathfrak{sl}_2,2n+1)$ above imply the Witt group relation
\begin{align*}
&&[\mathcal{C}(\mathfrak{so}_{2n+1},4)][\mathcal{C}(\mathfrak{sl}_2,2n+1)]&=[\mathcal{C}(\mathfrak{sp}(2(2n+1)),1)] \\
\Rightarrow&&[\mathcal{C}(\mathfrak{so}_{2n+1},4)_R^0]&=[(\mathcal{C}(\mathfrak{sl}_2,2n+1)^{\boxtimes2}_\mathrm{ad})^\mathrm{rev}]
\end{align*}
Hence \cite[Proposition 5.15(iii)]{DMNO} (which we have used as our definition) there exist connected \'etale algebras $A_1,A_2$ such that there is a braided equivalence
\begin{equation}\label{eq:so4}
(\mathcal{C}(\mathfrak{so}_{2n+1},4)_R^0)^0_{A_1}\simeq((\mathcal{C}(\mathfrak{sl}_2,2n+1)^{\boxtimes2}_\mathrm{ad})^\mathrm{rev})_{A_2}^0.
\end{equation}
But $(\mathcal{C}(\mathfrak{sl}_2,2n+1)^{\boxtimes2}_\mathrm{ad})^\mathrm{rev}$ is completely anisotropic.  In detail: $\mathcal{C}(\mathfrak{sl}_2,2n+1)_\mathrm{ad}$ is simple and completely anisotropic \cite[Section 5.5]{DNO}, and by the characterization of connected \'etale algebras in Deligne tensor products \cite[Theorem 3.6]{DNO}, any nontrivial connected \'etale algebra would arise from a braided equivalence $\mathcal{C}(\mathfrak{sl}_2,2n+1)_\mathrm{ad}\simeq\mathcal{C}(\mathfrak{sl}_2,2n+1)_\mathrm{ad}^\mathrm{rev}$, which does not exist \cite[Example 6.4(2)]{DMNO}.  Lastly, as $A_2=\mathbbm{1}$, then $A_1=\mathbbm{1}$ as well because both $\dim(\mathcal{C}(\mathfrak{so}_{2n+1},4)_R^0)$ and $\dim(\mathcal{C}(\mathfrak{sl}_2,2n+1)_\mathrm{ad})^2$  are equal to $(N^2/4)\csc^4(\pi/N)$ where $N=2n+3$.  This is computable directly from the quantum Weyl denominator formula \cite[Section 2.2.1]{COQUEREAUX2014258} (and sources within), and can also be realized via the rank-level duality \cite{NACULICH1990417} between $\mathfrak{so}(2n+1)_4$ and $\mathfrak{so}(4)_{2n+1}$ with the identification $\mathfrak{so}(4)\cong\mathfrak{su}(2)\times\mathfrak{su}(2)$.
\end{proof}
\begin{note} The fusion rules of these exceptional cases were incidentally studied in the language of ``$\mathbb{Z}/2\mathbb{Z}$ permutation gauging'' of $\mathcal{C}(\mathfrak{sl}_2,2n+1)_\mathrm{ad}$ based on low-level examples in  \cite{2018arXiv180401657E}.  Lemma \ref{lem:B} is a proof of their Conjecture 6.2.
\end{note}
\end{subsection}


\begin{subsection}{Type $C_n$ for $n\in\mathbb{Z}_{\geq3}$}\label{c}

\par The case of $\mathcal{C}(\mathfrak{sp}_{2n},1)$ is explained by the isomorphism of Grothendieck rings of $\mathcal{C}(\mathfrak{sp}_{2n},1)$ and $\mathcal{C}(\mathfrak{sl}_2,n)$ \cite[Theorem 5.1]{gannon2002}.  Moreover by computing the twist of the simple current, when $4$ divides $n$ there is a unique nontrivial connected \'etale algebra $R$ of Type D and $\mathcal{C}(\mathfrak{sp}_{2n},1)_R^0$ is simple \cite[Section 5.5]{DNO}.  The isomorphism of Grothendieck rings of $\mathcal{C}(\mathfrak{sp}_{2n},k)$ and $\mathcal{C}(\mathfrak{sp}_{2k},n)$ implies the $k=2$ cases were discussed in Section \ref{secB} as $\mathfrak{sp}_4\cong\mathfrak{so}_5$.  In particular $\mathcal{C}(\mathfrak{sp}_8,2)_R^0$ has exceptional fusion subcategories.

\par For all $n\in\mathbb{Z}_{\geq3}$, $\mathcal{C}(\mathfrak{sp}_{2n},k)_\text{pt}$ has $\mathbb{Z}/2\mathbb{Z}$ fusion.  The nontrivial simple current acts on the short dominant root $\beta:=\lambda_2$ by $\beta\mapsto\lambda_{n-2}+(k-1)\lambda_n$.  Hence $\overline{\beta}$ is simple for all $n\geq3$ if $k\geq3$.  The simple object of smallest dimension greater than 1 is the natural representation $\lambda_1$ and its orbit under the simple current \cite[Proposition, Section 4]{gannon1996}.  The argument provided by Gannon, Ruelle, \& Walton first provides a list of \emph{candidate} weights \cite[Lemma 2(3)]{gannon1996}, which for Type C are all fundamental weights, or specific multiples of fundamental weights defined in \cite[Equations 4.3(a),(b)]{gannon1996}, and then proceeds to determine which has smallest dimension generically.   As with $\beta=\lambda_2$, $\overline{\lambda_j}$ is simple for all fundamental weights $\lambda_j$ \cite[Section 3.3]{gannon2002}.  The only possible decomposable free module on a multiple of a fundamental weight would be $k\lambda_{n/2}$ when $n$ is even.  Thus if $\overline{\beta}$ factors, then $\lambda$ as in inequality (\ref{newcrux}) must have dimension larger than any candidate weight whose free module is not simple.  In particular $\dim(\beta)$ when $n$ is odd (we may choose any candidate in this case), or $\min\{\dim(\beta),\dim(k\lambda_{n/2})\}$ when $n$ is even.  Hence $\dim(\beta)\leq4$ when $n$ is odd, or $\min\{\dim(\beta),\dim(k\lambda_{n/2})\}\leq4$ when $n$ is even by the argument in (\ref{barg}).  This is false in $\mathcal{C}(\mathfrak{sp}_6,3)$ and $\mathcal{C}(\mathfrak{sp}_8,3)$, respectively and both $\dim(\beta)$ and $\dim(k\lambda_{n/2})$ are increasing functions of $n$ and $k$.
\end{subsection}


\begin{subsection}{Type $D_n$ for $n\in\mathbb{Z}_{\geq4}$}\label{D}

\par The categories $\mathcal{C}(\mathfrak{so}_{2n},1)$ are pointed.  The categories $\mathcal{C}(\mathfrak{so}_{2n},2)$ are known as the \emph{even metaplectic} categories and must be considered on a case-by-case basis, though most results exist in the literature.  First we consider $H=\mathbb{Z}/2\mathbb{Z}$  generated by $2\lambda_1$.  This is the unique order 2 subgroup when $n$ is odd and the diagonal $\mathbb{Z}/2\mathbb{Z}\subset\mathbb{Z}/2\mathbb{Z}\times\mathbb{Z}/2\mathbb{Z}$ when $n$ is even.  The categories $\mathcal{C}(\mathfrak{so}_{2n},2)_R^0$ are pointed with $\mathbb{Z}/2n\mathbb{Z}$ fusion \cite[Theorem 4.5]{2017arXiv171010284B}\cite[Theorem 3.8]{2016arXiv160904896B}.  When $n\equiv0\pmod{4}$, $\mathcal{C}(\mathfrak{so}_{2n},2)_\text{pt}\simeq\text{Rep}(\mathbb{Z}/2\mathbb{Z}\oplus\mathbb{Z}/2\mathbb{Z})$, hence there are additional pointed fusion subcategories.  When $H=Z_{\mathfrak{so}_{2n}}$, we can consider $\mathcal{C}(\mathfrak{so}_{2n},2)_R^0$ as the category of local modules over the regular algebra of the unique pointed fusion subcategory $\mathbb{Z}/2\mathbb{Z}\subset\mathbb{Z}/2n\mathbb{Z}$ described in the $|H|=2$ case.  The resulting cyclic group has dimension $2n/2^2=n/2$ and thus $\mathcal{C}(\mathfrak{so}_{2n},2)_R^0$ is pointed wtih $\mathbb{Z}/(n/2)\mathbb{Z}$ fusion.  Last of the $k=2$ cases are $H$ being the off-diagonal $\mathbb{Z}/2\mathbb{Z}\subset\mathbb{Z}/2\mathbb{Z}\times\mathbb{Z}/2\mathbb{Z}$ when $4$ divides $n$. Gustafson, Rowell, \& Ruan \cite[Theorem 2]{2018arXiv180800698G} have shown $\mathcal{C}(\mathfrak{so}_{2n},2)_R^0\simeq\mathcal{C}(\mathfrak{so}_{n/2},2)$ as fusion categories, whose fusion subcategories are under the umbrella of Sawin's classification \cite[Theorem 1]{Sawin06}.

\par Now for all $n\in\mathbb{Z}_{\geq4}$ the orbit of the dominant root $\beta:=\lambda_2$ under the action of the simple currents is $\lambda_2$, $(k-2)\lambda_1+\lambda_2$, $\lambda_{n-2}+(k-2)\lambda_n$, and $\lambda_{n-2}+(k-2)\lambda_{n-1}$. Thus $\overline{\beta}$ is simple when $k\geq3$.  We now use the reasoning of Section \ref{c}.  For Type D at generic levels the simple noninvertible object of least dimension is the natural representation $\lambda_1$ and its orbit under the simple currents and conjugations of Type D \cite[Section 3.4]{gannon2002}.  Once again all other candidate weights \cite[Lemma 2(4)]{gannon1996} are fundamental weights or multiples thereof (except there is no candidate which is a nontrivial multiple of $\lambda_n$).  The free modules on all fundamental weights are simple for $n\geq4$ and $k\geq3$  and there are only two possible nontrivial multiples of fundamental weights $\lambda$ for which $\overline{\lambda}$ is potentially decomposable: $(k/2)\lambda_1$ when $k$ is even, and $(k/2)\lambda_{n-1}$ when both $n$ and $k$ are even.

\par Therefore we have our desired constraints.  First if $n$ is odd and $k\geq3$ we must have $\min\{\dim(\beta),\dim((k/2)\lambda_1)\}\leq4,16$ depending on if $|H|=2,4$, respectively.  We compute $|H|=2$ only if $k$ is even and $|H|=4$ only if $8\mid k$.  From this we conclude there are no exceptional cases for $n$ odd because $\min\{\dim(\beta),\dim((k/2)\lambda_1)\}>5,18$ in $\mathcal{C}(\mathfrak{so}_{10},4)$ and $\mathcal{C}(\mathfrak{so}_{10},8)$, respectively, and the relevant dimensions are increasing functions of $n$ and $k$.  Secondly, when $n$ is even and $k\geq3$, define
\begin{equation}M:=\min\{\dim(\beta),\dim((k/2)\lambda_1),\dim((k/2)\lambda_{n-1})\}.
\end{equation}
Again we must have $M\leq4,16$ when $|H|=2,4$, respectively, and as before $|H|=2$ only if $k$ is even, but $|H|=4$ when $4\mid k$ by computing the twists of the simple currents.  There are no exceptional cases when $|H|=2$ because $M>5$ in $\mathcal{C}(\mathfrak{so}_{8},4)$ and the relevant dimensions are increasing functions of $n$ and $k$.  There are potential exceptional cases when $|H|=4$ as $M\leq16$ when $k=4$ and $n$ even is arbitrary, and $n=4$ with $k=8$.  We see this by noting $M>16$ in $\mathcal{C}(\mathfrak{so}_8,10)$ and $\mathcal{C}(\mathfrak{so}_{12},8)$, and the relevant dimensions are increasing functions of $n$ and $k$.  For $n=4$ with $k=8$, we compute using the quantum Weyl dimension formula that (in order), the smallest dimensions of simple objects in $\mathcal{C}(\mathfrak{so}_{8},8)$ are approximately $1.000$, $5.494$ and $14.592$.  The final dimension is $\dim(\beta)$ so it is an easy check that $\dim(\beta)$ cannot factor as in (\ref{qa}) using these dimensions.  We will now describe the exceptional cases when $k=4$ as in Section \ref{secB}.

\begin{lemma}\label{lemd2}
If $n\in\mathbb{Z}_{\geq4}$ is even and $R$ is the maximal (dimension 4) connected \'etale algebra of Type D in $\mathcal{C}(\mathfrak{so}_{2n},4)$, then
\begin{equation}
(\mathcal{C}(\mathfrak{so}_{2n},4)_R^0)^\mathrm{rev}\simeq(\mathcal{C}(\mathfrak{sl}_2,2n)_S^0)^{\boxtimes2}
\end{equation}
where $S$ is the unique nontrivial connected \'etale algebra of Type D in $\mathcal{C}(\mathfrak{sl}_2,2n)$.
\end{lemma}

\begin{proof}
This proof will follow the same logic as that of Lemma \ref{lem:B}.  Note that $\mathcal{C}(\mathfrak{sp}_{4n},1)\simeq\mathcal{C}(\mathfrak{sl}_2,2n)^\mathrm{rev}$ is a braided equivalence in this case, hence we have by the rank-level embedding
\begin{equation}
[\mathcal{C}(\mathfrak{so}_{2n},4)_R^0]=[\mathcal{C}(\mathfrak{sl}_2,2n)_S^0]^{-2}
\end{equation}
and our result follows as $(\mathcal{C}(\mathfrak{sl}_2,2n)_S^0)^{\boxtimes2}$ is completely anisotropic unless $n=8,14$ and has the same global dimension as $\mathcal{C}(\mathfrak{so}_{2n},4)_R^0$.  Our result is true for $n=8,14$ where the de-equivariantizations are rank 36 and 100, respectively as well.  There exists a fusion subcategory generated by either simple summand of $\overline{2\lambda_1}$, of ranks 6 and 10, respectively.  As there are no simple objects of integer dimension other than $\mathbbm{1}$, the fusion subcategories are modular and the fusion rules coincide with those of $(\mathcal{C}(\mathfrak{sl}_2,16)_S^0$ and $(\mathcal{C}(\mathfrak{sl}_2,28)_S^0$, respectively, where $S$ is the unique nontrivial connected \'etale algebra of Type D.
\end{proof}

\end{subsection}


\begin{subsection}{Types $E_6$ and $E_7$}\label{secE}

\par Simple objects of $\mathcal{C}(E_6,k)$ and $\mathcal{C}(E_7,k)$ will be indexed as in \cite[Section 3.5--3.6]{gannon2002}.  We compute $\mathcal{C}(E_6,k)_\mathrm{pt}$ is Tannakian if and only if $3\mid k$ so let $k\geq3$.  A simple current acts by cyclically permuting $\lambda_0\rightarrow\lambda_1\rightarrow\lambda_5$, and $\lambda_2\rightarrow\lambda_4\rightarrow\lambda_6$ so the dominant root $\lambda_6$ is permuted and moreover $\overline{\lambda_6}$ is simple.  The simple object of smallest dimension is is $\lambda_1$ \cite[Proposition (D)]{gannon1996} and its orbit under the simple currents and conjugation with $\dim(\lambda_1)=[17]+[9]+[1]$ and $\dim(\lambda_6)=[8][9][13][3]^{-1}[4]^{-1}$.
If $\dim(\lambda_1)^2/3^2>\dim(\lambda_6)$, then the same is true replacing $\lambda_1$ with any weight, hence $\overline{\lambda_6}$ cannot factor.  As $\dim(\lambda_1)$ is an increasing function of $k$, it suffices to find $k$ such that $\dim(\lambda_1)^2>3^2\cdot78$ since $\dim(\lambda_6)\leq78$.  This is true for $k\geq123$ and one verifies $\dim(\lambda_1)^2/3^2>\dim(\lambda_6)$ for $60\leq k<123$ as well.  It is then an elementary computation to verify $\overline{\lambda_6}$ cannot factor by dimension checks for the twenty cases $3\leq k\leq 60$ such that $3\mid k$.
%
We compute $\mathcal{C}(E_7,k)_\mathrm{pt}$ is Tannakian if and only if $4\mid k$ so let $k\geq4$.  The simple current acts by transposing $\lambda_0\leftrightarrow\lambda_6$, $\lambda_1\leftrightarrow\lambda_5$, and $\lambda_2\leftrightarrow\lambda_4$ so the dominant root $\lambda_1$ is permuted and moreover $\overline{\lambda_1}$ is simple.  Note the simple object of smallest dimension is $\lambda_6$ \cite[Proposition (E)]{gannon1996} with $\dim(\lambda_6)=[28]+[18]+[10]$ and $\dim(\lambda_1)=[12][14][19][4]^{-1}[6]^{-1}$.  As for $E_6$, $\overline{\beta}$ cannot factor $\dim(\lambda_6)^2>2^2\cdot133$ and $\dim(\lambda_1)\leq133$ when $k\geq15$.  And one verifies $\dim(\lambda_6)^2/2^2>\dim(\lambda_1)$ manually for $k=4,8,12$.

\end{subsection}
%

\end{proof}

%

\begin{corollary}\label{cor:one}
If $R$ is the regular algebra of $\mathrm{Rep}(Z_\mathfrak{g})\simeq\mathcal{C}(\mathfrak{g},k)_\text{pt}$, then $\mathcal{C}(\mathfrak{g},k)_R^0$ is simple aside from the exceptions in Theorem \ref{th:simplicity} (a)--(h).
\end{corollary}

\begin{proof}
Theorem \ref{th:simplicity} implies $(\mathcal{C}(\mathfrak{g},k)_R^0)_\mathrm{pt}\simeq\mathrm{Vec}$ in non-exceptional cases with relative commutant $\mathcal{C}(\mathfrak{g},k)_R^0$ and there are no other fusion subcategories. 
\end{proof}

\end{section}


\begin{section}{A classification of rank 2 Witt group relations}\label{sec:witt}


\begin{subsection}{$\mathcal{C}(\mathfrak{so}_5,k)$}\label{beetoo}

\par The numerical data for $\mathcal{C}(\mathfrak{so}_5,k)$ can be found in Section 2.3.3 of \cite{schopieray2}.  We will index the simple objects of $\mathcal{C}(\mathfrak{so}_5,k)$ in this section by ordered pairs $(s,t)$, with $s,t\in\mathbb{Z}_{\geq0}$ such that $s+t\leq k$.  The rank of $\mathcal{C}(\mathfrak{so}_5,k)$ is therefore the triangular number $(1/2)(k+1)(k+2)$.  For even $k$, $\mathcal{C}(\mathfrak{so}_5,k)_\text{pt}\simeq\text{Rep}(\mathbb{Z}/2\mathbb{Z})$ with unique nontrivial connected \'etale algebra $R$, otherwise $\mathcal{C}(\mathfrak{so}_5,k)_\text{pt}\simeq\mathrm{sVec}$.  Simple objects of $\mathcal{C}(\mathfrak{so}_5,k)_R$ are summands of free $R$-modules $(s,t)\otimes R=(s,t)\oplus(k-s-t,t)$ \cite[Lemma 3.4]{KiO}.   We compute \cite[Lemma 4]{Ostrik2003},
\begin{equation}
\dim\mathrm{Hom}_{\mathcal{C}_R}((s,t)\otimes R,(s,t)\otimes R)=\dim\mathrm{Hom}((s,t),(s,t)\otimes R)=1+\delta_{s,k-s-t}.\end{equation}
Thus the free module $(s,t)\otimes R$ is simple if and only if $s\neq k-s-t$.  Lemma 3.15 of \cite{DMNO} then implies the free $R$-modules are local if and only if $s$ is even (i.e. \!$(s,t)\in\mathcal{C}(\mathfrak{so}_5,k)_\mathrm{ad}$).  The rank of $\mathcal{C}(\mathfrak{so}_5,2n)_R^0$ for $n\in\mathbb{Z}_{\geq1}$ is thus $(n+1)(n+4)/2$.  We will study lower levels in detail where exceptional behavior occurs, then describe their general behavior at all other levels.

\paragraph{Level 1} This is an example of an Ising category \cite[Appendix B]{DGNO}.  As explained in \cite[Section 6.4(3)]{DMNO}, there exists a unique odd number $\ell$, $1\leq\ell\leq15$ such that $[\mathcal{C}(\mathfrak{so}_5,1)]=[\mathcal{C}(\mathfrak{sl}_2,2)]^\ell$.  Computing $\xi(\mathcal{C}(\mathfrak{so}_5,1))=\exp(5\pi i/8)$ determines
\begin{align}
[\mathcal{C}(\mathfrak{so}_5,1)]&=[\mathcal{C}(\mathfrak{sl}_2,2)]^7,\text{ and} \\
[\mathcal{C}(\mathfrak{so}_5,1)]^{16}&=[\text{Vec}].
\end{align}

\paragraph{Level 2} The category $\mathcal{C}(\mathfrak{so}_5,2)_R^0$ is pointed with $\mathbb{Z}/5\mathbb{Z}$ fusion.  All relations in $\mathcal{W}_\mathrm{pt}$ are known \cite[Section 5.3]{DMNO} and by computing central charge $\xi(\mathcal{C}(\mathfrak{so}_5,2))=-1$ we have
\begin{equation}
[\mathcal{C}(\mathfrak{so}_5,2)]^2=[\text{Vec}].
\end{equation}
Note the conformal embedding $B_{2,2}\subset A_{4,1}$ \cite[Section 3.3]{coq2} achieves this result as well.

\paragraph{Level 3} The conformal embedding $B_{2,3}\subset D_{5,1}$ gives $[\mathcal{C}(\mathfrak{so}_5,3)]=[\mathcal{C}(\mathbb{Z}/4\mathbb{Z},q)]$ where $q(1)=q(3)=\exp(5\pi i/4)$ and $q(2)=-1$.  We compute $\xi(\mathcal{C}(\mathbb{Z}/4\mathbb{Z},q))=-(1+i)/\sqrt{2}$ and thus
\begin{equation}
[\mathcal{C}(\mathfrak{so}_5,3)]^8=[\text{Vec}].
\end{equation}

\paragraph{Level 4}  Theorem \ref{th:simplicity}(c) implies
\begin{equation}
[\mathcal{C}(\mathfrak{so}_5,4)]=[\mathcal{C}(\mathfrak{so}_5,4)_R^0]=[(\mathcal{C}(\mathfrak{sl}_2,5)_\mathrm{ad}^\text{rev})^{\boxtimes2}]=[\mathcal{C}(\mathfrak{sl}_2,5)_\mathrm{ad}]^{-2}. \label{witt:so54}
\end{equation}
In particular, $[\mathcal{C}(\mathfrak{so}_5,4)]$ has infinite order in $\mathcal{W}$.  Furthermore, $\mathcal{C}(\mathfrak{sl}_2,5)_\text{pt}\simeq\mathcal{C}(\mathfrak{sl}_2,1)$ hence
\begin{equation}
[\mathcal{C}(\mathfrak{so}_5,4)][\mathcal{C}(\mathfrak{sl}_2,5)]^2=[\mathcal{C}(\mathfrak{sl}_2,1)]^2.
\end{equation}

\paragraph{Level 7} The conformal embedding $B_{2,7}\subset D_{7,1}$ gives
\begin{equation}
[\mathcal{C}(\mathfrak{so}_5,7)]=[\mathcal{C}(\mathbb{Z}/4\mathbb{Z},q)].
\end{equation}
A central charge computation verifies $\xi(\mathcal{C}(\mathfrak{so}_5,7))=\exp(7\pi i/4)$ hence
\begin{equation}
[\mathcal{C}(\mathfrak{so}_5,7)]^3=[\mathcal{C}(\mathfrak{so}_5,3)].
\end{equation}

\paragraph{Level 12} The conformal embedding $B_{2,12}\subset E_{8,1}$ gives
\begin{equation}
[\mathcal{C}(\mathfrak{so}_5,12)]=[\text{Vec}].
\end{equation}

\paragraph{Generic levels} Note that $\mathcal{C}(\mathfrak{so}_5,k)_\text{ad}$ is the fusion subcategory with simple objects indexed by weights in the root lattice by Sawin's classification and moreover $\mathcal{C}(\mathfrak{so}_5,k)_\text{ad}$ are $s$-simple when $k$ is odd and completely anisotropic by the announced work of Gannon \cite{banffgannon} classifying all modular invariants for $B_2$.  The non-existence of nontrivial modular invariants implies the non-existence of connected \'etale algebras as explained in \cite[Section 4.1]{schopieray2017}. These categories are slightly degenerate by the double-centralizer theorem \cite[Theorem 3.2(i)]{mug1}.  As explained in Section \ref{sec:wittdef}, the only relations which could occur between $[\mathcal{C}(\mathfrak{so}_5,k)_\mathrm{ad}]$ are if one of these representative categories is equivalent to its reverse.  Any braided equivalence $F:\mathcal{C}(\mathfrak{so}_5,k)_\mathrm{ad}\stackrel{\sim}{\to}\mathcal{C}(\mathfrak{so}_5,k)_\mathrm{ad}^\text{rev}$ induces an automorphism of the Grothendieck ring, which were classified in \cite[Theorem 3.C]{gannon2002}.  For $k\geq13$, all automorphisms of the Grothendieck ring are trivial, therefore $F$ must be a gauge autoequivalence, i.e. \!one which acts trivially on objects.  We compute $\theta(0,2)=\exp(3/(k+3)\cdot2\pi i)$.  The full twist in the reverse category is inverse to the original, so for $F$ to exist we must have $\theta(0,2)=\pm1$ which only happens when $k=3$.  For even $k$, $\mathcal{C}(\mathfrak{so}_5,k)_R^0$ is simple except $k=2,4$ by Theorem \ref{th:simplicity}.

\begin{theorem}
All relations in $\mathcal{W}$ in the subgroup generated by $[\mathcal{C}(\mathfrak{so}_5,k)]$ are generated by
\begin{align}
[\mathcal{C}(\mathfrak{so}_5,1)]^{16}&=[\textnormal{Vec}], \\
[\mathcal{C}(\mathfrak{so}_5,2)]^2&=[\textnormal{Vec}], \\
[\mathcal{C}(\mathfrak{so}_5,3)]^8&=[\textnormal{Vec}], \\
[\mathcal{C}(\mathfrak{so}_5,7)]^3&=[\mathcal{C}(\mathfrak{so}_5,3)],\text{ and} \\
[\mathcal{C}(\mathfrak{so}_5,12)]&=[\textnormal{Vec}].
\end{align}
\end{theorem}

\begin{proof}
We have collected simple, completely anisotropic, nondegenerate fusion categories
\begin{equation}
\mathcal{C}(\mathfrak{so}_5,2m)_R^0\qquad\qquad\text{for }m\in\mathbb{Z}_{\geq3}\text{ and }m\neq6,\label{dyslecticeq}
\end{equation}
and $s$-simple, completely anisotropic, slightly nondegenerate fusion categories
\begin{equation}
\mathcal{C}(\mathfrak{so}_5,2m+1)_\mathrm{ad}\qquad\qquad\text{for }m\in\mathbb{Z}_{\geq2}\text{ and }m\neq3.\label{centralizereq}
\end{equation}
The categories of (\ref{dyslecticeq}) and (\ref{centralizereq}) are pairwise non-equivalent as the former are modular and the latter supermodular.  Categories of (\ref{dyslecticeq}) and (\ref{centralizereq}) are (internally) pairwise non-equivalent as they are all of distinct ranks.  Also observe the slightly degenerate images of the categories of (\ref{dyslecticeq}) in s$\mathcal{W}$ of the form $\mathcal{C}\boxtimes\text{sVec}$ are $s$-simple and completely anisotropic as $\mathcal{C}(\mathfrak{so}_5,2m)_R^0$ are unpointed for $m\geq3$ and $m\neq6$.  For those categories in (\ref{dyslecticeq}), either their images have infinite order in s$\mathcal{W}$ or $[\mathcal{C}(\mathfrak{so}_5,2m)_R^0]^2=[\text{Vec}]$.  This would necessitate $\xi(\mathcal{C}(\mathfrak{so}_5,2m)_R^0)=\pm1$.  We compute $\xi(\mathcal{C}(\mathfrak{so}_5,2m)_R^0)=\xi(\mathcal{C}(\mathfrak{so}_5,2m))=\exp(5m/(2m+3)\pi i)$, thus $\xi(\mathcal{C}(\mathfrak{so}_5,2m)_R^0)=\pm1$ if and only if $5m/(2m+3)\in\mathbb{Z}$.  But $2<5m/(2m+3)<5/2$ for $m>6$, leaving $m=1,6$ as the only candidate cases, which were discussed above.  These, along with $[\mathcal{C}(\mathfrak{so}_5,4)]$, are the generators of the torsion free part of the subgroup of $\mathcal{W}$ generated by the classes $[\mathcal{C}(\mathfrak{so}_5,k)]$.  The classes of finite order lie in $\mathcal{W}_\mathrm{pt}$ and so their relations are known.
\end{proof}

\end{subsection}


\begin{subsection}{$\mathcal{C}(\mathfrak{g}_2,k)$}\label{g22}
\begin{theorem}
All relations in $\mathcal{W}$ in the subgroup generated by $[\mathcal{C}(\mathfrak{g}_2,k)]$ are generated by
\begin{align}
[\mathcal{C}(\mathfrak{g}_2,3)]^4&=[\textnormal{Vec}],\text{ and} \\
[\mathcal{C}(\mathfrak{g}_2,4)]^8&=[\textnormal{Vec}].
\end{align}
\end{theorem}

\begin{proof}
There are conformal embeddings $G_{2,3}\subset E_{6,1}$ and $G_{2,4}\subset D_{7,1}$.  Thus $[\mathcal{C}(\mathfrak{g}_2,3)]$ and $[\mathcal{C}(\mathfrak{g}_2,4)]$ lie in $\mathcal{W}_\mathrm{pt}$ where all relations are known.  Now note that all $\mathcal{C}(\mathfrak{g}_2,k)$ are simple \cite[Theorem 1]{Sawin06}, and for $k\neq3,4$ these categories are completely anisotropic due to a complete classification of modular invariants for $\mathfrak{g}_2$ the aforementioned work of Gannon \cite{banffgannon}.  The images of $\mathcal{C}(\mathfrak{g}_2,k)$ under the map $\mathcal{W}\to s\mathcal{W}$ via $\mathcal{C}\mapsto\text{sVec}\boxtimes\mathcal{C}$ are $s$-simple as $\mathcal{C}(\mathfrak{g}_2,k)$ are unpointed.  Therefore any relation in $s\mathcal{W}$ amongst $[\mathcal{C}(\mathfrak{g}_2,k)]$ is of the form $[\mathcal{C}(\mathfrak{g}_2,k)]=[\mathcal{C}(\mathfrak{g}_2,k)]^{-1}$ \cite[Remark 5.11]{DNO} which implies $\xi(\mathcal{C}(\mathfrak{g}_2,k))=\pm1$.  We compute $\xi(\mathcal{C}(\mathfrak{g}_2,k))=\exp(7k/(2(k+4))\pi i)$, thus $\xi(\mathcal{C}(\mathfrak{g}_2,k))=\pm1$ only if $k=4m$ for some $m\in\mathbb{Z}_{\geq0}$.  This is possible only if $7m/(m+1)\in\mathbb{Z}_{\geq1}$, thus $m=6$.  Lastly we verify that $\mathcal{C}(\mathfrak{g}_2,24)\not\simeq\mathcal{C}(\mathfrak{g}_2,24)^\text{rev}$ by noting there are no nontrivial automorphisms of the Grothendieck ring \cite[Theorem 5.1]{gannon1996} which necessitates $\theta(1,0)=\theta(1,0)^{-1}$ and this is false.  Moreover $[\mathcal{C}(\mathfrak{g}_2,k)]$ has infinite order unless $k=3,4$.
\end{proof}
As explained in \cite[Section 6.4(8)]{DMNO} via conformal embedding, $[\mathcal{C}(\mathfrak{g}_2,1)]=[\mathcal{C}(\mathfrak{sl}_2,28)]$, both having infinite order in $\mathcal{W}$.  One could also note that $\mathcal{C}(\mathfrak{sl}_2,3)_\mathrm{ad}\simeq\mathcal{C}(\mathfrak{g}_2,1)$ hence any Witt group relation involving $\mathcal{C}(\mathfrak{g}_2,1)$ will be described by a classification of relations in Type A.  Similarly $\mathcal{C}(\mathfrak{g}_2,2)$ is rank 4, and one verifies using the classification of rank 4 modular tensor categories \cite{rowell2009}, that $\mathcal{C}(\mathfrak{g}_2,2)$ is a Galois conjugate of $\mathcal{C}(\mathfrak{sl}_2,7)_\mathrm{ad}$.  Via conformal embedding $A_{1,7}\times G_{2,2}\subseteq E_{7,1}$,
\begin{align}
&&[\mathcal{C}(\mathfrak{sl}_2,7)][\mathcal{C}(\mathfrak{g}_2,2)]=[\mathcal{C}(E_7,1)]&=[\mathcal{C}(\mathfrak{sl}_2,1)]^{-1} \nonumber\\
\Rightarrow && [\mathcal{C}(\mathfrak{sl}_2,1)][\mathcal{C}(\mathfrak{sl}_2,7)][\mathcal{C}(\mathfrak{g}_2,2)]&=[\mathrm{Vec}].
\end{align}
\end{subsection}

\begin{subsection}{All Witt group relations for rank $\leq2$}\label{wittrel}

\par \par First we compare the completely anisotropic representative lists above for $\mathfrak{so}_5$ and $\mathfrak{g}_2$ with those in (4.1) through (4.3) of \cite{schopieray2017} for $\mathfrak{sl}_3$ and those at the end of Section 5.5 of \cite{DNO} for $\mathfrak{sl}_2$, accumulated here for ease for all $m\in\mathbb{Z}_{\geq1}$.
\begin{align}
&\mathcal{C}(\mathfrak{g}_2,m) &&\text{for }m\geq5 \label{g2a}\\[0.2cm]
&\mathcal{C}(\mathfrak{so}_5,2m)_R^0 &&\text{for }m\neq1,2,6 \label{so5a}\\
&\mathcal{C}(\mathfrak{so}_5,2m+1)_\mathrm{ad} &&\text{for }m\neq1,3 \label{so5b}\\[0.2cm]
&\mathcal{C}(\mathfrak{sl}_3,3m+1)_\mathrm{ad} && \label{sl3a}\\
&\mathcal{C}(\mathfrak{sl}_3,3m+2)_\mathrm{ad} &&\text{for }m\neq1\text{ and }m=0 \label{sl3b}\\
&\mathcal{C}(\mathfrak{sl}_3,3m)_R^0 &&\text{for }m\neq1,3,7 \label{sl3c}\\[0.2cm]
&\mathcal{C}(\mathfrak{sl}_2,2m+1)_\mathrm{ad} &&  \label{sl2a}\\
&\mathcal{C}(\mathfrak{sl}_2,4m+2)_\mathrm{ad} &&\text{for }m\neq1,2\label{sl2b}\\
&\mathcal{C}(\mathfrak{sl}_2,4m)_R^0 &&\text{for }m\neq1,2,7\label{sl2c}
\end{align}

\par We are searching for any categories in the above lists which have a braided equivalence between them, or a braid-reversing equivalence between them.  Any such occurence will be called a \emph{coincidence}, for brevity.  Note that $7/2>7\ell/(2\ell+8)>3$ for $5\leq\ell\leq25$ and $5/2>5m/(2m+3)>2$ for $m\geq7$, so in this range categories from (\ref{g2a}) and (\ref{so5a}) are not coincidental by comparing central charge.  There is only one coincidence of rank in the unaccounted for range: $\mathcal{C}(\mathfrak{g}_2,7)$ and $\mathcal{C}(\mathfrak{so}_5,8)_R^0$, but using the above formulas, their central charges are neither equal nor inverse to one another.  Hence lists (\ref{g2a}) and (\ref{so5a}) contain no coincidences, while both contain no coincidences when compared to list (\ref{so5b}) as the former are simple.  For this reason, the only coincidences involving list (\ref{so5b}) would be with list (\ref{sl2b}) (and vice-versa).  But the object of minimal dimension in either case is that which corresponds to the shortest dominant root, with dimension $[3]$ in the case of $\mathfrak{sl}_2$ and of $[5][6][2]^{-1}[3]^{-1}$ in the case of $\mathfrak{so}_5$.  For $m=2$, $\dim(1,0)>3$ in $\mathcal{C}(\mathfrak{so}_5,2m+1)_\mathrm{ad}$, and $[5][6][2]^{-1}[3]^{-1}$ is a strictly increasing function of $m\in\mathbb{Z}_{\geq2}$ (approaching $5$ as $m\to\infty$).  Since $[3]\leq3$ there are no coincidences between lists (\ref{so5b}) and (\ref{sl2b}).

\par Next note that list (\ref{g2a}) has no coincidences with (\ref{sl2a})--(\ref{sl2c}) as $\mathcal{C}(\mathfrak{sl}_2,k)$ for $k\in\mathbb{Z}_{\geq1}$ are multiplicity-free, while $\dim\text{Hom}((0,2)\otimes(0,2),(0,2))=2$ in $\mathcal{C}(\mathfrak{g}_2,k)$ with $k\geq3$.  The cases $k<5$ were discussed in Section \ref{g22}.  The categories $\mathcal{C}(\mathfrak{g}_2,k)$ for $k\geq5$ have no non-trivial fusion ring automorphisms \cite[Theorem 3]{gannon2002} which implies that there are no coincidences with the categories in (\ref{sl3a})--(\ref{sl3c}), except in the case when these categories are self-dual.  This is the case for $\mathcal{C}(\mathfrak{sl}_3,2)_\mathrm{ad}$ and $\mathcal{C}(\mathfrak{sl}_3,6)_R^0$, whose ranks are too small to merit a coincidence.

\par It remains to compare list (\ref{so5a}) with lists (\ref{sl3a})--(\ref{sl2c}).  One useful fact is that $\mathcal{C}(\mathfrak{so}_5,2m)_R^0$ for $m\geq3$ have at least two pairs of objects with equal dimensions.  The categories in (\ref{sl2a}) have distinct dimensions and those in (\ref{sl2c}) have exactly one pair of objects with the same dimension.  For comparing with lists (\ref{sl3a})--(\ref{sl3c}) we proceed by central charge computation using \cite[Corollary 3.4.1]{schopieray2017}.  For $\ell\geq3$, $3/2>9\ell/(6\ell+8)>1$, hence there are no coincidences between (\ref{sl3a}) and (\ref{so5a}) except possibly involving $\mathcal{C}(\mathfrak{sl}_3,4)_\mathrm{ad}$ or $\mathcal{C}(\mathfrak{sl}_3,7)_\mathrm{ad}$.  By comparing ranks, only the former is feasible, having the same rank as $\mathcal{C}(\mathfrak{so}_5,2)_R^0$, but this category is pointed, while $\mathcal{C}(\mathfrak{sl}_3,4)_\mathrm{ad}$ is not.  For $m\geq2$, $2>2m/(m+10)>1$, hence there are no coincidences between (\ref{sl3c}) and (\ref{so5a}) by computing central charge.  We end with the most finicky of arguments, showing there are no coincidences between lists (\ref{sl3b}) and (\ref{so5a}).  We do so by noting there are at least $(1/2)(n^2+n+4)$ self-dual objects in $\mathcal{C}(\mathfrak{so}_5,2n)_R^0$, disregarding whether the $2n$ decomposable simple $R$-modules are self-dual or not.  Comparatively, there are at most $k/2+1$ self-dual objects in $\mathcal{C}(\mathfrak{sl}_3,k)_\mathrm{ad}$.  Hence for a fixed $n\in\mathbb{Z}_{\geq1}$, to have $(1/2)(n^2+n+4)$ self-dual objects in $\mathcal{C}(\mathfrak{sl}_3,k)_\mathrm{ad}$ we must consider $k\geq n(n+1)$.  But the rank of $\mathcal{C}(\mathfrak{sl}_3,3k+2)_\mathrm{ad}$ is $(1/2)k(3k+1)$.  Hence for a fixed $n$, the corresponding categories from (\ref{sl3b}) with at least the same number of self-dual objects as $\mathcal{C}(\mathfrak{so}_5,2n)_R^0$, have rank at least $(1/2)n(n+1)(3n^2+3n+1)$ which is greater than $(1/2)(n+1)(n+4)$, the rank of $\mathcal{C}(\mathfrak{so}_5,2n)_R^0$ for all $n\in\mathbb{Z}_{\geq1}$.  Moreover there are no coincidences between (\ref{sl3b}) and (\ref{so5a}).

\par In summary, all non-trivial relations are described in Sections \ref{beetoo} and \ref{g22} above, \cite[Section 5.5]{DNO}, and \cite[Theorems 3,4]{schopieray2017}.  The list of known relations is sufficiently long that beyond rank 2 it will not be productive to compile the entire classification of relations once it is discovered/proven.

\end{subsection}

\begin{section}{Observations, problems and questions}\label{quest}

\par There are several approachable areas of research to explore using the results of this paper.
\begin{problem}
Classify all degenerate fusion subcategories of $\mathcal{C}(\mathfrak{g},k)_R^0$.\end{problem}
By Theorem \ref{th:simplicity}, these degenerate fusion categories will only appear when $\mathfrak{g}$ is of Type A or D, and by computing twists of simple currents, the list of cases can be reduced substantially.  The obstacle is that the dimensions of simple objects are less helpful without explicit $\boxtimes$-factorization.  But Davydov, Nikshych, and Ostrik \cite{DNO} describe $\boxtimes_\mathcal{E}$, a \emph{tensor product over a symmetric fusion subcategory} $\mathcal{E}$, attributed to Greenough \cite{greenough2010monoidal}.  In particular there is an analog to M\" uger's decomposition for braided fusion categories which are \emph{nondegenerate over $\mathcal{E}$}, i.e. \!$\mathcal{C}'\simeq\mathcal{E}$.  Their Proposition 4.3 states that if $\mathcal{E}\subset\mathcal{C}$ is any symmetric fusion subcategory, $\mathcal{C}$ is nondegenerate over $\mathcal{E}$, and $\mathcal{D}\subset\mathcal{C}$ is a fusion subcategory which is also nondegenerate over $\mathcal{E}$, then $\mathcal{D}\boxtimes_\mathcal{E}C_\mathcal{C}(\mathcal{D})\simeq\mathcal{C}$ is a braided equivalence.  The rank-level embedding $\mathfrak{so}(2n)_4\times\mathfrak{su}(2)_{2n}\subseteq\mathfrak{sp}(2(2n))_1$ for $n\in\mathbb{Z}_{\geq4}$ used in Lemma \ref{lem:B} gives interesting examples of this construction.

\begin{lemma}\label{lemd}
If $n\in\mathbb{Z}_{\geq7}$ is odd and $R\in\mathcal{C}(\mathfrak{so}_{2n},4)_\mathrm{pt}$ is the unique nontrivial connected \'etale algebra of Type D, then
\begin{equation}\label{lem:d}
\mathcal{C}(\mathfrak{so}_{2n},4)_R^0\simeq(\mathcal{C}(\mathfrak{sl}_2,2n)^\mathrm{rev}\boxtimes\mathcal{C}(\mathfrak{sp}_{4n},1))_S^0
\end{equation}
is a braided equivalence, where $S$ is the unique nontrivial connected \'etale algebra of Type D in $\mathcal{C}(\mathfrak{sl}_2,2n)^\mathrm{rev}\boxtimes\mathcal{C}(\mathfrak{sp}_{4n},1)$.
\end{lemma}

\begin{proof}
We will use the same basic proof as in Lemma \ref{lem:B}.  In this case $\mathcal{C}(\mathfrak{sp}_{4n},1)$ and $\mathcal{C}(\mathfrak{sl}_2,2n)$ are equivalent as fusion categories (not as modular tensor categories), but both are prime.  Hence the Witt group relation from the rank-level embedding only implies
\begin{equation}
[\mathcal{C}(\mathfrak{so}_{2n},4)_R^0]=[\mathcal{C}(\mathfrak{sl}_2,2n)^\mathrm{rev}\boxtimes\mathcal{C}(\mathfrak{sp}_{4n},1)],
\end{equation}
and moreover there exist connected \'etale algebras $A_1,A_2$ such that
\begin{equation}
(\mathcal{C}(\mathfrak{so}_{2n},4)_R^0)_{A_1}^0\simeq(\mathcal{C}(\mathfrak{sl}_2,2n)^\mathrm{rev}\boxtimes\mathcal{C}(\mathfrak{sp}_{4n},1))_{A_2}^0
\end{equation}
is a braided equivalence.  Here the category $\mathcal{C}(\mathfrak{sl}_2,2n)^\mathrm{rev}\boxtimes\mathcal{C}(\mathfrak{sp}_{4n},1)$ is not completely anisotropic.  We claim that it has a unique nontrivial connected \'etale algebra when $n\neq5$.  Indeed $\mathcal{C}(\mathfrak{sl}_2,2n)^\mathrm{rev}$ is completely anisotropic if $n\neq5$ by the ADE classification so any nontrivial connected \'etale algebra in the product would arise from a connected \'etale algebra $A_3\in\mathcal{C}(\mathfrak{sp}_{4n},1)$ and an equivalence of fusion subcategories $\mathcal{C}\simeq\mathcal{D}^\mathrm{rev}$, where $\mathcal{C}$ is a fusion subcategory of $\mathcal{C}(\mathfrak{sl}_2,2n)^\mathrm{rev}$ and $\mathcal{D}$ is a fusion subcategory of $\mathcal{C}(\mathfrak{sp}_{4n},1)^0_{A_3}$.  But $\dim(\mathcal{C}(\mathfrak{sl}_2,2n))=\dim(\mathcal{C}(\mathfrak{sp}_{4n},1))$ and the fusion subcategories of each factor are known, thus the only possibility is $A_3=\mathbbm{1}$ with the equivalence $\mathcal{C}(\mathfrak{sl}_2,2n)^\mathrm{rev}_\mathrm{ad}\simeq\mathcal{C}(\mathfrak{sp}_{4n},1)_\mathrm{ad}^\mathrm{rev}$.  The non-trivial connected \'etale algebra is evidently $S$, the product of invertible objects of twist -1.  Comparing global dimensions implies our result as before.
\end{proof}

\par The benefit of Lemma \ref{lemd} is ease of computation.  Not many simple $R$-modules are free but almost all simple $S$-modules are free, which allows the use of the monoidal structure of the free $S$-module functor to be used with the existing knowledge of the fusion rules of $\mathcal{C}(\mathfrak{sl}_2,k)$.

\begin{corollary}
If $n\in\mathbb{Z}_{\geq7}$ is odd, then there is a braided equivalence
\begin{equation}
(\mathcal{C}(\mathfrak{so}_{2n},4)_R^0)_\mathrm{ad}^\mathrm{rev}\simeq\mathcal{C}(\mathfrak{sl}_2,2n)_\mathrm{ad}^{\boxtimes_\mathrm{sVec}2}.
\end{equation}
\end{corollary}

\begin{proof}
The category $\mathcal{C}(\mathfrak{so}_{2n},4)_R^0$ is nondegenerate (over $\mathrm{Vec}$), but $(\mathcal{C}(\mathfrak{so}_{2n},4)_R^0)_\mathrm{ad}$ is nondegenerate over $\mathrm{sVec}$ (slightly degenerate, or supermodular).  The fusion subcategories $\otimes_S$-generated by the simple free $S$-modules $\overline{(2)\boxtimes(0)}$ and $\overline{(0)\boxtimes(2)}$ are equivalent to $\mathcal{C}(\mathfrak{sl}_2,2n)_\mathrm{ad}^\mathrm{rev}$, and $\mathcal{C}(\mathfrak{sp}_{4n},1)_\mathrm{ad}\simeq\mathcal{C}(\mathfrak{sl}_2,2n)_\mathrm{ad}^\mathrm{rev}$, respectively.  These are known to be slightly degenerate and our result follows from \cite[Proposition 4.3]{DNO}.
\end{proof}

\par We conjecture that a similar, but twisted result is true when $n$ is even.

\begin{conjecture}
If $n\in\mathbb{Z}_{\geq4}$ is even, and $R$ is the regular algebra of the ``diagonal'' Tannakian subcategory of $\mathcal{C}(\mathfrak{so}_{2n},4)$ of rank 2, then there is a braided equivalence
\begin{equation}
(\mathcal{C}(\mathfrak{so}_{2n},4)_R^0)_\mathrm{ad}^\mathrm{rev}\simeq\mathcal{C}(\mathfrak{sl}_2,2n)_\mathrm{ad}^{\boxtimes_{\mathrm{Rep}(\mathbb{Z}/2\mathbb{Z})}2}.
\end{equation}
\end{conjecture}

\par There are other promising $\boxtimes_\mathcal{E}$-factorizations of adjoint subcategories in Type A.

\begin{conjecture}
Let $R$ be the regular algebra of the unique nontrivial Tannakian subcategory of $\mathcal{C}(\mathfrak{sl}_{4},4)$.  There is a braided equivalence
\begin{equation}
(\mathcal{C}(\mathfrak{sl}_{4},4)_R^0)_\mathrm{ad}\simeq\mathcal{C}(\mathfrak{sl}_2,6)_\mathrm{ad}^{\boxtimes_\mathrm{sVec}2}.
\end{equation}
\end{conjecture}

\par For even $n\in\mathbb{Z}_{\geq}4$, there exists a Tannakian fusion subcategory of rank $n/2$.  If we denote the regular algebra of this fusion subcategory by $R$, then $(\mathcal{C}(\mathfrak{sl}_n,n)_R^0)_\mathrm{pt}\simeq\mathrm{sVec}$.

\begin{question}
Does $(\mathcal{C}(\mathfrak{sl}_n,n)_R^0)_\mathrm{ad}$ always have a $\boxtimes_\mathrm{sVec}$-factorization?
\end{question}

\par The reason for this line of inquiry is to study relations in the Witt group of slightly degenerate braided fusion categories $s\mathcal{W}$.  For instance we note that $\mathcal{C}(\mathfrak{sl}_2,6)_\mathrm{ad}\simeq\mathcal{C}(\mathfrak{sl}_2,6)_\mathrm{ad}^\mathrm{rev}$ \cite[Section 5.5]{DNO}, hence if true, Conjecture 2 implies
\begin{equation}
(\mathcal{C}(\mathfrak{sl}_{4},4)_R^0)_\mathrm{ad}\simeq\mathcal{C}(\mathfrak{sl}_2,6)_\mathrm{ad}\boxtimes_\mathrm{sVec}\mathcal{C}(\mathfrak{sl}_2,6)_\mathrm{ad}^\mathrm{rev}\simeq\mathcal{Z}(\mathcal{C}(\mathfrak{sl}_2,6)_\mathrm{ad},\mathrm{sVec}),
\end{equation}
by \cite[Corollary 4.4]{DNO} where $\mathcal{Z}(\mathcal{C}(\mathfrak{sl}_2,6)_\mathrm{ad},\mathrm{sVec})$ is the centralizer of $\mathrm{sVec}\hookrightarrow\mathcal{Z}(\mathcal{C}(\mathfrak{sl}_2,6)_\mathrm{ad})$.

\begin{problem}
Classify all relations in $s\mathcal{W}$ generated by supermodular subcategories of $\mathcal{C}(\mathfrak{g},k)_R^0$ when $\mathfrak{g}$ is Type A or D.
\end{problem}

\end{section}

\end{section}


{\footnotesize\bibliography{bib}}
\bibliographystyle{plain}

\end{document}